\documentclass[11pt,reqno]{amsart}
\usepackage{amsmath, amsfonts, amsthm, amssymb}

\theoremstyle{plain}
\newtheorem{theorem}{Theorem}[section]
\newtheorem{lemma}{Lemma}[section]
\newtheorem{proposition}{Proposition}[section]

\theoremstyle{definition}

\theoremstyle{remark}
\newtheorem{remark}{Remark}[section]

\theoremstyle{example}

\theoremstyle{corollary}
\newtheorem{corollary}{Corollary}[section]

\numberwithin{equation}{section}

\renewcommand{\div}{\textnormal{div }}
\newcommand{\R}{{\mathbb R}}

\begin{document}
\title[The Navier-Stokes Equations with Vorticity Boundary Condition]{The Navier-Stokes
Equations with\\ the Kinematic and Vorticity Boundary Conditions\\
on Non-Flat Boundaries}
%\title[Boundary Conditions for the Navier-Stokes Equations]{On Boundary
%Conditions on Non-Flat Boundaries\\ for the Navier-Stokes Equations}
\author{Gui-Qiang Chen \and Dan Osborne \and Zhongmin Qian}
\address{G.-Q. Chen,  School of Mathematical Sciences, Fudan University,
 Shanghai 200433, China; Department of Mathematics, Northwestern University,
         Evanston, IL 60208-2730, USA}
\email{\tt gqchen@math.northwestern.edu}
\address{D. Osborne, Mathematical Institute, University of Oxford, 24-29 St Giles, Oxford OX1 3LB, UK}
\email{\tt dan.osborne@st-annes.ox.ac.uk}
\address{Z. Qian, Mathematical Institute, University of Oxford, 24-29 St Giles, Oxford OX1 3LB, UK}
\email{\tt qianz@maths.ox.ac.uk}

\dedicatory{To Professor Wenjun Wu with Affection and Admiration}

%Dedicated to Professor Wenjun Wu on the Occasion of His 90th Birthday}

\keywords{Navier-Stokes equations, incompressible, vorticity
boundary condition, kinematic boundary condition, absolute boundary
condition, non-flat boundary, general domain, Stokes operator,
Neumann problem, Poisson equations, vorticity, strong solutions,
inviscid limit, slip boundary condition} \subjclass[2000]{Primary:
35Q30,76D05,76D03,35A35,35B35; Secondary: 35P05,76D07,36M60,58C35}
\date{\today}
%\thanks{}
%

\begin{abstract}
We study the initial-boundary value problem of the Navier-Stokes
equations for incompressible fluids in a general domain in $\R^n$
with compact and smooth boundary, subject to the kinematic and
vorticity boundary conditions on the non-flat boundary. We observe
that, under the nonhomogeneous boundary conditions, the pressure $p$
can be still recovered by solving the Neumann problem for the
Poisson equation. Then we establish the well-posedness of the
unsteady Stokes equations and employ the solution to reduce our
initial-boundary value problem into an initial-boundary value
problem with absolute boundary conditions. Based on this, we first
establish the well-posedness for an appropriate local linearized
problem with the absolute boundary conditions and the initial
condition (without the incompressibility condition), which
establishes a velocity mapping. Then we develop \emph{apriori}
estimates for the velocity mapping, especially involving the Sobolev
norm for the time-derivative of the mapping to deal with the
complicated boundary conditions, which leads to the existence of the
fixed point of the mapping and the existence of solutions to our
initial-boundary value problem. Finally, we establish that, when the
viscosity coefficient tends zero, the strong solutions of the
initial-boundary value problem in $\R^n (n\ge 3)$ with
nonhomogeneous vorticity boundary condition converges in $L^2$ to
the corresponding Euler equations satisfying the kinematic
condition.
\end{abstract}
\maketitle

\section{Introduction}

The motion of an incompressible viscous fluid
% confined in a bounded domain $\Omega\subset\mathbb{R}^{n}$
in $\R^n, n\ge 2$, is described by the Navier-Stokes equations:
\begin{eqnarray}
&&\partial_t u+ u\cdot\nabla u +\nabla p=\mu \Delta u, \label{1.1}\\
&&\nabla\cdot u = 0, \label{1.2}
\end{eqnarray}
with initial data
\begin{equation}
u|_{t=0}=u_{0}(x), \label{1.3}
\end{equation}
where $u(t,x)=(u^{1}, \cdots, u^{n})(t,x)$ is the velocity vector
field and $p(t,x)$ is the pressure that maintains the
incompressibility of a fluid
at $(t,x)$. Equation \eqref{1.2}, i.e. $\div u=0$, is the
incompressibility condition. As a nonlinear system of partial
differential equations, $u$ and $p$ are regarded as unknown
functions, and the initial velocity field $u_{0}(x)$
sets the fluid in motion. The constant $\mu>0$ is the kinematic
viscosity constant, $u\cdot\nabla$ denotes the covariant derivative
along the flow trajectories, namely, the directional derivative in
the direction $u$, $\Delta {u}$ is the usual Laplacian on $u$, and
$\mu \Delta u$ represents the stress applied to the fluid. As usual,
we use $\nabla\cdot=\div$ to denote the divergence operator.

For inviscid flow, $\mu=0$ and then the equations are referred to as
the Euler equations for incompressible fluid flow:
\begin{eqnarray}
&&\partial_t u+ u\cdot\nabla u +\nabla p=0, \label{1.4}\\
&&\nabla\cdot u = 0. \label{1.5}
\end{eqnarray}

When a fluid is confined in a bounded domain
$\Omega\subset\mathbb{R}^{n}$ with non-empty boundary $\Gamma$,
these equations must be supplied with proper boundary conditions in
order to be well-posed. For concreteness, the bounded domain
$\Omega\subset\R^n$ is assumed to have a compact, oriented, smooth
surface boundary $\Gamma$.
In this paper, we propose and study the initial-boundary value
problem \eqref{1.1}--\eqref{1.3} with the following boundary
conditions on the non-flat boundary $\Gamma$:
\begin{eqnarray}
&&\left. u^\bot\right|_{\Gamma}=0  \qquad \text{(kinematic
condition)},
\label{kinematic-1}\\
&& \left. \omega^{\Vert}\right|_{\Gamma}=a \qquad \text{(vorticity
condition)}, \label{vorticity-1}
\end{eqnarray}
where the field $a=a(t,x)$ is defined on the non-flat boundary
$\Gamma$, $u^\bot\in \R$ denotes the normal component and $u^\Vert$
the tangential part of $u\in \R^n$.

Various physical considerations for some flows observed in the
nature have led to the no-slip condition:
\begin{equation}\label{no-slip-1}
\left. u^\Vert\right|_{\Gamma}=0,
\end{equation}
together with the kinematic condition \eqref{kinematic-1}. This is
equivalent to the simple homogeneous Dirichlet boundary condition:
$u|_{\Gamma}=0$.
Thus, it has been received an intensive investigation for the
Navier-Stokes equations.

\medskip
On the other hand, if a vector field $u\in\R^n$ is identified with
its corresponding differential $1$-form on $\Omega$, then various
boundary problems for the linearized equations, the heat equations,
have been studied in geometry under the Hodge theory on manifolds
with boundaries; see Conner \cite{conner1956} and the references
cited therein. In particular, three natural boundary conditions have
been identified; since
\begin{equation*}
\Delta u=-\nabla \times \omega + \nabla (\div u)
\qquad\,\,\text{on}\,\,\Omega\subset\mathbb{R}^{n},
\end{equation*}
where $\omega:=\nabla\times u\in \R^n$ is the vorticity, boundary
conditions may be imposed on the components ${u}^{\bot}$,
${u}^{\Vert}$, $\omega^{\bot}$, $\omega^{\Vert}$, and $\div u$. It
has been revealed that the initial-boundary value problem with the
boundary conditions \eqref{kinematic-1}--\eqref{vorticity-1} for the
heat equations is well-posed.
This leads to a natural question whether the same setup of
initial-boundary value problem is well-posed for the Navier-Stokes
equations. In the literature, the \emph{absolute boundary
conditions} mean that our boundary conditions
\eqref{kinematic-1}--\eqref{vorticity-1} with $a=0$, i.e.
homogeneous.

The initial-boundary value problem \eqref{1.1}--\eqref{1.3} and
\eqref{kinematic-1}--\eqref{vorticity-1} for the Navier-Stokes
equations has been received less attention, partly due to the common
agreement that the no-slip condition is a physical condition under
ideal situations. However, recent experimental evidence (cf.
Einzel-Panzer-Liu \cite{Einzel-Panzer-Liu}, Lauga-Brenner-Stone
\cite{lauga-etc}, and Zhu-Garnick \cite{zhu-Garnick} for a survey)
demonstrates that the velocity field of a fluid does not satisfy the
complete no-slip conditions in general. Indeed, it is revealed that
the pressure does depend on how curved the surface boundary is; also
see Bellout-Neustpua-Penel \cite{bellout-Neustpua-Penel}. It is also
evident that, for high speed flows, one can not expect to comply the
complete no-slip conditions. Therefore, a careful study of
incompressible fluids with various boundary conditions associated
with the Navier-Stokes equations is in great demand.

Another motivation for our study in this paper is from the Navier
slip boundary condition, which are widely accepted in many
applications and numerical studies. It states that the slip velocity
is proportional to the shear stress:
\begin{equation}\label{Navier-1}
u^\Vert=-\zeta \big((\nabla u+(\nabla u)^\top)\nu\big)^\Vert,
\end{equation}
where $\nu$ is the outward normal on $\Gamma$ and $\zeta>0$ is the
slip length on $\Gamma$. Such boundary conditions can be induced by
effects of free capillary boundaries, a rough boundary,  a
perforated boundary, or an exterior electric field (cf.
\cite{APV,Bansch,BJ,CDE,JM-1,JM-2,QWS,Schwartz}).
Take an asymptotic expansion of $u(t,x)$ with respect to $\zeta$ and
the distance $d(x,\Gamma)$ to $\Gamma$ such that $d(x,\Gamma)\ll
\zeta$. Then, by a direct calculation, one finds that the Navier
slip condition \eqref{Navier-1} yields \eqref{vorticity-1} in which
$a(t,x)$ is determined by the initial value problem of a family of
linear heat equations on $\Gamma$, that is, $a(t,x)$ is determined
solely by the initial data \eqref{1.3}.

\medskip
There is another appealing reason to consider other boundary
problems associated with the Navier-Stokes equations. The Euler
equations \eqref{1.4}--\eqref{1.5}
subject to the kinematic condition \eqref{kinematic-1}, along with
the initial condition \eqref{1.3},
are often regarded as an ideal model in turbulence theory, where the
Reynolds number is so big that the viscosity coefficient $\mu$ is
negligible. Hence, a natural question is whether the flow determined
by the Euler equations is, at best, a ``singular limit'' of the
Navier-Stokes equations
as the viscosity coefficient $\mu \downarrow 0$ (cf. Pope
\cite{pope2000}, page 18). However, it is very subtle and difficult
in general, since physical boundary layers may be present. The
classical no-slip boundary condition, $u|_{\Gamma}=0$, give arise to
the phenomenon of strong boundary layers in general as formally
derived by Prandtl \cite{Prandtl}. For some sufficient conditions to
ensure the convergence of viscous solutions to the ones of the Euler
equations, see Kato \cite{Kato,Wangxm} and the references therein.
However, as far as we know, such a claim has never been proved for
general bounded domains with the Navier boundary condition. In
contrary, we prove in this paper that
our problem \eqref{1.1}--\eqref{1.3} and
\eqref{kinematic-1}--\eqref{vorticity-1}
for the Navier-Stokes equations converges in $L^2$  to a solution of
problem \eqref{1.4}--\eqref{kinematic-1} for the Euler equations as
the viscosity coefficient $\mu \downarrow 0$.
To explain why the Euler equations, subject to the kinematic
condition \eqref{kinematic-1}, is the limiting case of the solutions
$u_{\mu}$ of our problem \eqref{1.1}--\eqref{1.3} and
\eqref{kinematic-1}--\eqref{vorticity-1},
we
notice
that the pressure $p$ required in the Euler equations solves the
Poisson equation:
\begin{equation}
\Delta p=-\nabla\cdot\left(u\cdot\nabla u\right), \qquad \left.
\partial_\nu p\right|_{\Gamma }=\pi (u, u), \label{1.7}
\end{equation}
where $u(t,x)$ is a solution to the Euler equations
\eqref{1.4}--\eqref{1.5}, and $\pi$ is the second fundamental form
of the boundary $\Gamma$ which describes the curvature of the
boundary surface $\Gamma$. Since $u$ does not necessarily comply the
no-slip condition, the normal derivative of the pressure $p$ does
not vanish along the boundary in general. In particular, the
pressure $p$ along the boundary naturally depends on the curvature
of the boundary. On the other hand, under the kinematic condition
\eqref{kinematic-1}
the pressure $p$ can be recovered from the Neumann problem of the
Poisson equation:
\begin{eqnarray}
&&\Delta p=-\nabla \cdot (u_{\mu}\cdot\nabla u_{\mu}), \label{1.9}\\
&&\left. \partial_\nu p\right|_{\Gamma}
  =\pi(u_\mu, u_\mu)-\mu \nabla^{\Gamma}\times\omega_\mu^{\Vert}.
      \label{eq-a2-12-10-07}
\end{eqnarray}
As $\mu\downarrow 0$, the pressure $p$ satisfies the same equation
as that of the pressure required by the Euler equations. However, if
we further enforce the no-slip condition, then the normal derivative
of the pressure $p$ vanishes, which does not in general coincide
with that of the Euler equations. On the other hand, if we set a
free condition on the tangent component of $u$ but replace the
boundary condition on the vorticity, we have the chance to recover
the right boundary condition for the pressure in the Euler
equations.

The rigorous mathematical analysis of the Navier-Stokes equations
involving vorticity-type boundary conditions may date back the work
by Solonnikov-{\v{S}}{\v{c}}adilov \cite{SS} for the stationary
linearized Navier-Stokes system under boundary conditions:
\begin{equation}\label{ss}
\left. u^{\bot}\right| _{\Gamma }=0, \qquad \big((\nabla u+(\nabla
u)^\top)\nu\big)^\Vert|_\Gamma=0.
\end{equation}
For two-dimensional, simply connected, bounded domains, the
vanishing viscosity problem has been rigorously justified by
Yodovich \cite{Yo}. We also refer to Lions \cite{Lions1,Lions2} for
the boundary conditions \eqref{kinematic-1}--\eqref{vorticity-1}
with $a=0$, Clopeau, Mikeli\'{c}, and Robert \cite{CMR} and Lopes
Filho, Nussenzveig Lopes and Planas \cite{LLP} for the Navier
boundary condition \eqref{ss}, and Mucha \cite{Mucha} under some
geometrical constraints on the shape of the domains for the
two-dimensional case. Also see \cite{XX} for an analysis for the
complete slip boundary conditions as \eqref{absolute-1} below.

\medskip
The main purpose of this paper is to develop an approach to deal
with the well-posedness and the inviscid limit for the
initial-boundary value problem \eqref{1.1}--\eqref{1.3} and
\eqref{kinematic-1}--\eqref{vorticity-1} with {\em nonhomogeneous
boundary condition} for the multidimensional Navier-Stokes equations
($n\ge 3$). One of the main difficulties for the higher dimension
case, in comparison with the two-dimensional case, is that the
maximum principle for the vorticity fails so that the
two-dimensional techniques can not be directly extended to this
case. Furthermore, the nonhomogeneous boundary condition causes
another main difficulty in developing apriori estimates which
require to be compatible with the nonlinear convection term. In
Section 2, we introduce the local moving frames on the boundary
$\Gamma$ of a domain $\Omega$ and derive some basic identities for
vector fields.
Although we work with a domain in the Euclidean space, the boundary
$\Gamma$ is a curved surface, so that geometric tools have been
brought in to carry out local computations. Since we can work only
with the boundary coordinate system which may be not ``normal'', the
computations we need to carry out are a little bit complicated and
long. The results in Section 2 allow us to determine boundary values
of several interesting quantities, which will be used to settle the
local existence of a unique strong solution for problem
\eqref{1.1}--\eqref{1.3} and
\eqref{kinematic-1}--\eqref{vorticity-1} for the Navier-Stokes
equations.

In Section 3, we establish various $L^2$-estimate for vector fields
that satisfy the absolute boundary conditions:
\begin{equation}\label{absolute-1}
\left. u^{\bot}\right|_{\Gamma}=0, \qquad \left.
\omega^{\Vert}\right|_{\Gamma }=0,
\end{equation}
In particular, the main novel fact is that, if $u(t,x)$ satisfies
the absolute boundary conditions \eqref{absolute-1}, then
$$
\left. (\nabla \times \omega)^{\bot}\right|_{\Gamma}=0.
$$
The latter information on
\begin{equation*}
\nabla \times \omega=\Delta u- \nabla (\div u)
\end{equation*}
along the boundary is crucial in obtaining the necessary
\emph{apriori} estimates for the linearization of the absolute
boundary problem.

In Sections 4, we develop an old idea of Leray-Hopf
\cite{hopf1,leray2,leray3,leray1} for our initial-boundary value
problem with nonhomogeneous vorticity boundary condition for the
Navier-Stokes equations. Namely, under the nonhomogeneous boundary
conditions, the pressure $p$ may be recovered by solving the Neumann
problem \eqref{1.9}--\eqref{eq-a2-12-10-07} of the Poisson
equations. Then we establish the well-posedness of the unsteady
Stokes equations and employ the solution $w$ to reduce our
initial-boundary value problem into an initial-boundary value
problem with absolute boundary conditions for $v=u-w$. This leads to
the following linearlization: One first considers the linear
parabolic equation:
\begin{equation}
\partial_t v+ (\beta+w)\cdot \nabla (v+w) + \nabla p_{\beta}=\mu \Delta v,
\label{eq09-0013}
\end{equation}
where $p_{\beta}$ solves (\ref{1.7}) replacing $u$ by $\beta+w$ for
a given $\beta=\beta(t,x), t\leq T$, satisfying the absolute
conditions and the initial condition. Based on this, in Section 5,
we demonstrate that the absolute boundary conditions are suitable
boundary conditions, since they generally do not appear in the
literature on parabolic equations,  and conclude that
(\ref{eq09-0013}) is well-posed, which thus establishes a velocity
mapping $V: \, w\rightarrow v$. Note that we have dropped the
incompressibility condition, so that (\ref{eq09-0013}) is local and
a linear parabolic equation. However, we will establish that any
fixed point in a proper functional space is a strong solution to the
Navier-Stokes equations, so that the incompressibility can be
recovered.

We emphasize that the absolute boundary conditions
in a general domain are a kind of boundary conditions which change
from the Dirichlet to the Neumann boundary condition along the
boundary. Therefore, the coercive estimates, or called the global
$L^{p}$-estimates, for parabolic equations are not applicable in
this setting. Indeed, we have to establish estimates quite close to
the coercive estimates for the linear systems (\ref{eq09-0013}) in a
different functional space, which will be done in Section 6. In
order to construct a local (in time) strong solution, we need to
develop \emph{apriori} estimates for the solution $v(t,x)$ of the
linear parabolic equation (\ref{eq09-0013}). It would be enough to
develop an estimate for the $H^{2}$-norm of $v(t,x)$ in terms of
that of $(\beta+w)(t,x)$. In fact, it works for the Dirichlet
boundary problem (which has been done by Solonnikov
\cite{Solonnikov1973}) and also works for the periodic case (cf.
\cite{kreiss-lorenz1}). Unfortunately, due to the complicated
boundary conditions, we are unable to achieve such \emph{apriori}
estimates, instead, we have to involve the Sobolev norm for the
time-derivative of $v(t,x)$.

Finally, in Sections 8--9, we establish that the unique strong
solution of the initial-boundary value problem
\eqref{1.1}--\eqref{1.3} and
\eqref{kinematic-1}--\eqref{vorticity-1} in $\R^n, n\ge 3$, with
nonhomogeneous vorticity boundary condition converges in $L^2$ to
the corresponding Euler equations satisfying the kinematic condition
\eqref{kinematic-1}.

\section{Local Moving Frames on the Boundaries and Basic Identities
for Vector Fields}

Let $\Omega\subset\R^n$ be a bounded domain with a compact, smooth
boundary $\Gamma=\partial\Omega$ with $n\ge 2$. We assume that
$\Gamma$ is an oriented surface (not necessarily connected) which
carries an induced metric. Let $\mathbf{\nu}$ denote the unit normal
vector field along $\Gamma$ pointing outwards. As usual, we use
conventional notation that the repeated indices in a formula are
understood to be summed up from $1$ to $n$ unless confusion may
occur.

In order to reflect better the geometry of the boundary
$\Gamma=\partial\Omega$ of $\Omega$, it is convenient to work with
local moving frames compatible to $\Gamma$ (see Palais-Terng
\cite{Palais-Terng} for the details). More precisely, by a moving
frame compatible to the boundary we mean a local orthonormal basis
$$
\{e_{1},\cdots ,e_{n}\}
$$
of the tangent bundle $T\Omega $ of $\Omega$ about a boundary point
$x\in\Gamma$ such that $e_{n}=\mathbf{\nu}$ when restricted to
$\Gamma$, where $\mathbf{\nu}$ is the unit normal vector field along
$\Gamma$ pointing outwards.

The Christoffel symbols are defined by
$\Gamma_{ij}^{l}=\theta_{j}^{l}(e_{i})$. In terms of the Christoffel
symbols:
\begin{equation*}
\nabla_{e_{i}}e_{j}=\Gamma_{ij}^{l}e_{l}, \qquad
\nabla_{e_{i}}\alpha^{j}=-\Gamma_{il}^{j}\alpha^{l},
\end{equation*}
where $\nabla_{X}$ means the directional derivative in $X$. The
torsion-free condition may be stated as
$\Gamma_{ij}^{k}=-\Gamma_{ik}^{j}$ and,
along the boundary $\Gamma=\partial\Omega$,
\begin{equation*}
\Gamma_{ij}^{n}=\Gamma_{ji}^{n}\equiv -h_{ij}, \qquad \text{for
any}\,\, 1\le i,j\leq n-1,
\end{equation*}
and $\pi =(h_{ij})$ is the second fundamental form, which is a
symmetric tensor on $\Gamma$. By definition,
$$
h_{ij}=\langle \nabla_{e_{i}} \mathbf{\nu }, e_{j}\rangle.
$$
If $u=\sum_{i=1}^{n-1}u^{i}e_{i}
$ and $w=\sum_{i=1}^{n-1}w^{i}e_{i}$ are two vector fields tangent to $%
\Gamma $, then
\begin{equation*}
\pi (u,w)=\sum_{i,j=1}^{n-1}h_{ij}u^{i}w^{j} .
\end{equation*}

The surface $\Gamma$ carries an induced metric, its Levi-Civita
connection is denoted by $\nabla^{\Gamma }$. Then $\{e_{1}, \cdots,
e_{n-1}\}$ restricted to $\Gamma$ is a moving frame of $T\Gamma$. If
$u$ is a vector field on $\Omega$, then the tangent part of $u$
restricted to $\Gamma$, denoted by $u^{\Vert}$, is a section of
$T\Gamma$, and its covariant derivative is denoted by
$\nabla^{\Gamma}u^{\Vert}$.
Let $H$ denote the mean curvature of
$\Gamma$, that is, $H$ is the trace of the second fundamental form
$\pi$:
$$
H=\sum_{j=1}^{n-1}h_{jj}.
$$

\begin{lemma}\label{lems1}
Let $u$ and $w$ be two vector fields on $\Omega$. Then
\begin{eqnarray}
\langle w\cdot\nabla u,\, \mathbf{\nu }\rangle &=&-\pi(u^{\Vert},
w^{\Vert})-H\langle w,\, \nu\rangle \langle u,\, \nu\rangle +\langle
w, \,\nu\rangle (\nabla\cdot u)\notag\\
&&+\langle w^{\Vert}, \nabla^{\Gamma}\langle u, \nu \rangle \rangle
+\langle u^{\Vert },\nabla^{\Gamma}\langle w, \mathbf{\nu}\rangle
\rangle -\nabla^{\Gamma}\cdot \big(\langle w,\, \mathbf{\nu}\rangle
\,u^{\Vert}\big). \label{june21-045}
\end{eqnarray}
In particular, if $u^{\bot}=w^{\bot}=0$ on $\Gamma$, then
\begin{equation}
\langle w\cdot \nabla u,\, \nu \rangle =-\pi(u,w).
\label{june21-052}
\end{equation}
\end{lemma}

\begin{proof}
Since
\begin{equation*}
\left(w\cdot\nabla u\right)^{i}=
w^{j}\big(e_{j}(u^{i})+\Gamma_{jk}^{i}u^{k}\big),
\end{equation*}
then
\begin{eqnarray}
\langle w\cdot\nabla u,\, \nu \rangle
&=&w^{j}\big(e_{j}(u^{n})+\Gamma_{jk}^{n}u^{k}\big)\notag  \\
&=&\langle w^{\Vert}, \nabla^{\Gamma}\langle u,\, \nu \rangle
\rangle -\pi(u^{\Vert}, w^{\Vert}) +w^{n}(\nabla_{n}u^{n}).
\label{2.10}
\end{eqnarray}
On the other hand, according to the Ricci equation,
\begin{eqnarray*}
\sum_{j=1}^{n-1}\nabla_{j}u^{j}=\sum_{j=1}^{n-1}\nabla_{j}^{\Gamma}u^{j}
+\sum_{j=1}^{n-1}h_{jj}\langle u, \, \nu\rangle
=\nabla^{\Gamma}\cdot u^{\Vert}+H\langle u,\, \nu \rangle,
\end{eqnarray*}
so that
\begin{eqnarray}
\nabla_{n}u^{n}=\nabla\cdot u-\sum_{j=1}^{n-1}\nabla_{j}u^{j}
=\nabla\cdot u-\nabla^{\Gamma}\cdot u^{\Vert}
  -H\langle u, \nu\rangle. \label{2.11}
\end{eqnarray}
Substitution \eqref{2.11} into \eqref{2.10} yields
\begin{eqnarray*}
\langle w\cdot\nabla u,\, \nu \rangle =\langle w^{\Vert},
\nabla^{\Gamma} \langle u, \nu\rangle \rangle -\pi(u^{\Vert},
w^{\Vert}) +\langle w, \nu\rangle\big(\nabla\cdot
u-\nabla^{\Gamma}\cdot u^{\Vert}-H\langle u, \nu\rangle\big),
\end{eqnarray*}
which, together with the identity:
\begin{equation*}
\langle w, \nu\rangle\nabla^{\Gamma}\cdot u^{\Vert}
=\nabla^{\Gamma}\cdot \big(\langle w, \nu\rangle u^{\Vert}\big)
-\langle u^{\Vert}, \nabla^{\Gamma}\langle w, \nu\rangle\rangle,
\end{equation*}
yields \eqref{june21-045}.
\end{proof}

\begin{lemma}\label{lems2}
Let $u$ be a vector field on $\Omega$, $\omega=\nabla\times u$, and
$d=\nabla\cdot u$. Then

{\rm (i)} $\left. \partial_\nu d\right|_{\Gamma}=\langle \Delta u,\,
\nu \rangle +\nabla^{\Gamma}\times\omega^{\Vert}$, where
\begin{equation*}
\nabla^{\Gamma}\times\omega^{\Vert}=\nabla_{1}^{\Gamma}\omega^{2}
-\nabla_{2}^{\Gamma}\omega^{1}
\end{equation*}
is independent of the choice of a local moving frame, which can be
identified with the exterior derivative of $\omega^{\Vert}$ on
$\Gamma$;

{\rm (ii)} $\frac{1}{2}\partial_\nu(|u|^{2})|_{\Gamma}=\langle
u\times\omega, \, \nu\rangle +\langle u, \,\nu\rangle d
-\pi(u^{\Vert}, u^{\Vert})
-H\langle u, \,\nu\rangle^{2}
+2\langle u^{\Vert}, \nabla^{\Gamma}\langle u, \nu \rangle \rangle$

$\qquad\qquad\qquad\quad\,\,$ $-\nabla^{\Gamma}\cdot\big(\langle
u,\, \nu\rangle u^{\Vert}\big)$.
\end{lemma}

\begin{proof} (i)
According to the vector identity:
\begin{equation*}
\Delta u=-\nabla \times \omega +\nabla d,
\end{equation*}
one obtains
\begin{equation*}
\left. \partial_\nu d\right|_{\Gamma} =\langle\Delta u,\, \nu\rangle
+\langle \nabla \times \omega, \,\nu\rangle.
\end{equation*}
It remains to verify that $\langle\nabla\times\omega, \,\nu\rangle$
coincides with $\nabla^{\Gamma}\times\omega^{\Vert}$ which depends
only on the tangent part $\left. \omega^{\Vert}\right|_{\Gamma}$.
The last fact follows easily from a local computation.

(ii) This formula follows directly from the vector identity:
\begin{equation*}
\frac{1}{2}\nabla (|u|^{2})=u\times \left( \nabla \times u\right)
+\left(u\cdot \nabla \right) u,
\end{equation*}
and Lemma \ref{lems1}.
\end{proof}

\section{$L^{2}$-estimates for Vector Fields Satisfying the Absolute Boundary Conditions}

In this section, we establish $L^{2}$-estimates for vector fields
satisfying the absolute boundary conditions \eqref{absolute-1}.

Let $\Omega\subset\mathbb{R}^{3}$ be a bounded domain with a
compact, smooth boundary $\Gamma$. The Sobolev spaces can be defined
in terms of the total derivative $\nabla $. If $T$ is a (smooth)
vector field on $\Omega$ and $\nabla^{j}T$ is the $j$-th covariant
derivative, then, for a nonnegative integer $k$, the Sobolev norm
$\|\cdot\|_{W^{k,p}}, p\ge 1$, is given by
\begin{equation}
\|T\|_{W^{k,p}}=\big(\sum_{j=0}^{k}\|\nabla^{j}T\|_{L^{p}}^{p}\big)^{1/p},
\label{utc01}
\end{equation}
where $\|\cdot\|_{L^{p}}$ denotes the $L^{p}$-norm of vector fields,
that is, $\|T\|_{L^{p}}=(\int_{\Omega}|T|^{p}dx)^{1/p}$ with
$|T|=\sqrt{\langle T,T\rangle}$ the length of the vector field.
When $p=2$, we often use $H^k(\Omega)=W^{k,2}(\Omega)$ and
$\|T\|_{H^k}=\|T\|_{W^{k,2}}$, since such spaces are Hilbert spaces.
The most important in the theory of Sobolev spaces is the Sobolev
imbedding theorem (see Theorem 4.12 of Adams-Fournier
\cite{Adams-Fournier03}, page 85).

We will use the following integration by parts formula:
\begin{equation}
\int_{\Omega}\langle \nabla \times u,\, w\rangle\, dx
=\int_{\Omega}\langle u,\nabla \times w\rangle\, dx
+\int_{\Gamma}\langle u\times w,\mathbf{\nu}\rangle \, dS,
\label{s28-08}
\end{equation}
where $dS$ is the surface measure on $\Gamma$. This formula follows
easily from the Stokes theorem.

If $u$ is a vector field on $\Omega$, then three $L^{2}$-norms are
related to the total derivative $\nabla u$:

(i) The (squared) $L^{2}$-norm $\|\nabla u\|_{2}^{2}$;

(ii) $\|\nabla\cdot u\|_{2}^{2}+\|\nabla \times u\|_{2}^{2}$, which
is most useful in the study of the Navier-Stokes equations, since,
if $u$ is the velocity field of a incompressible fluid, then
$\nabla\cdot u=0$ and $\nabla \times u=\omega$ is the vorticity in
physics;

(iii) The quantity $-\int_{\Omega}\langle \Delta u, u\rangle$ which
is related to the spectral gap of the (Hodge De Rham) Laplacian
$\Delta$ which has a clear geometric meaning.

The relations among these three quantities are the following.

\begin{lemma}\label{les29-001}
If $u\in H^{2}(\Omega)$ is a vector field, then
\begin{eqnarray}
\int_{\Omega}\langle\Delta u, u\rangle\, dx
 &=&-\int_{\Omega}|\omega|^{2}\,dx
 -\int_{\Omega}|\nabla\cdot u|^{2}\,dx
 \notag\\ &&
 +\int_{\Gamma}\langle u\times \omega, \mathbf{\nu }\rangle\, dS
+\int_{\Gamma}\left(\nabla\cdot u\right) \langle
u,\mathbf{\nu}\rangle\,dS, \label{07june03-04}
\end{eqnarray}
and
\begin{eqnarray}
\int_{\Omega}|\nabla u|^{2}\,dx&=&\int_{\Omega}|\omega|^{2}\,dx
+\int_{\Omega}|\nabla\cdot u|^{2}\,dx
-\int_{\Gamma}\pi\big(u^{\Vert}, u^{\Vert}\big)\,dS\notag \\
&&-\int_{\Gamma}H|u^{\bot}|^{2}\, dS + 2\int_{\Gamma }\langle
u^{\Vert },\nabla ^{\Gamma }\langle u,\mathbf{\nu }\rangle
\rangle\,dS.
\label{s28-09}
\end{eqnarray}
\end{lemma}

\begin{proof}
Formula \eqref{07june03-04} follows from the vector identity
\begin{equation}
\Delta u=-\nabla\times\omega +\nabla (\nabla\cdot u)
\label{june03-01}
\end{equation}
and integration by parts formula (\ref{s28-08}). To show
(\ref{s28-09}), integrating the Bochner's identity:
\begin{equation}
\langle \Delta u,u\rangle =\frac{1}{2}\Delta |u|^{2}-|\nabla u|^{2},
\label{Bochner}
\end{equation}
together with integration by parts formulas, we obtain
\begin{eqnarray}
\int_{\Omega}\langle \Delta u,u\rangle &=&-\int_{\Omega}|\nabla
u|^{2}\,dx + \frac{1}{2}\int_{\Gamma}\partial_\nu(|u|^{2})\, dS
\notag \\
&=&-\int_{\Omega}|\nabla u|^{2}\, dx +\int_{\Gamma}\langle
u^{\Vert}\times \omega^{\Vert}, \nu \rangle \, dS
+\int_{\Gamma}\langle u, \mathbf{\nu}\rangle \nabla\cdot u\, dS  \notag \\
&&+2\int_{\Gamma}\langle u^{\Vert}, \nabla^{\Gamma}\langle u,
\mathbf{\nu}\rangle \rangle\, dS -\int_{\Gamma}\pi \big(u^{\Vert},
u^{\Vert}\big)\, dS -\int_{\Gamma}H\langle u,
\mathbf{\nu}\rangle^{2} \, dS, \label{07june03-03}
\end{eqnarray}
where we have used Lemma \ref{lems2}(ii). Now (\ref{s28-09}) follows
from (\ref{07june03-04}) and (\ref{07june03-03}).
\end{proof}

\begin{lemma}
\label{ths28-01} Suppose that $u\in H^{1}(\Omega)$ is a vector field
satisfying the boundary conditions that $\left.
u^{\Vert}\right|_{\Gamma_{1}}=0$ and $\left.
u^{\bot}\right|_{\Gamma_{2}}=0$, where $\Gamma =\Gamma_{1}\cup
\Gamma_{2}$. Then
\begin{equation}
\|u\|_{H^{1}}\leq C\|(\omega, \nabla\cdot u, u)\|_{2}
\label{07june06-01}
\end{equation}
for some constant $C$ depending only on $\Omega$.
\end{lemma}

\begin{proof}
According to (\ref{s28-09}) and the boundary conditions, we have
\begin{eqnarray*}
\int_{\Omega}|\nabla u|^{2} \, dx&=&\int_{\Omega}|(\omega,
\nabla\cdot u)|^{2}\, dx
-\int_{\Gamma_{2}}\pi(u^{\Vert}, u^{\Vert})\, dS
-\int_{\Gamma_{1}}H\langle u, \nu \rangle^{2}\, dS \\
&\leq &\int_{\Omega}|(\omega, \nabla\cdot u)|^{2}\, dx
+C\int_{\Gamma}|u|^{2} \, dS.
\end{eqnarray*}
Now the conclusion follows from the trace imbedding theorem (Theorem
1.5.1.10 in Grisvard \cite{gri1}):
For any
$\varepsilon\in (0,1]$,
\begin{equation}\label{3.8a}
\int_{\Gamma}|u|^{2}\, dS\leq \varepsilon \int_{\Omega}|\nabla
u|^{2}\, dx +\frac{C}{\varepsilon}\int_{\Omega}|u|^{2}\, dx,
\end{equation}
where $C$ is a constant depending only on $\Omega $.
\end{proof}

\begin{remark}
The above lemma holds not only for the $L^{2}$-norm, but also for
the $L^{p}$-norm for any $p\in \lbrack 2,\infty)$. Indeed, under the
same boundary conditions on $u$ as in Lemma \ref{ths28-01}, as a
special case of Theorem 10.5 of \cite{agmon1965}, we have
\begin{equation}
\|u\|_{W^{1,p}}\leq C\big(\|(\omega, \nabla\cdot
u)\|_{p}+\|u\|_{2}\big), \label{07june01-12}
\end{equation}
where $C$ also depends on $p\geq 2$. Estimate (\ref{07june01-12}) is
the Agmon-Douglis-Nirenberg's estimate (cf. \cite{agmon1965}).
\end{remark}

Next we consider the second derivative $\nabla^2 u$ of $u$.
\begin{lemma}\label{ths29-01}
There exists $C>0$ depending only on $\Omega$ such that
\begin{equation*}
\|\nabla^2 u\|_{2}^{2}\leq C\big(\|\Delta u\|_{2}^{2}+\|\nabla
u\|_{2}^{2}\big)
\end{equation*}
for any vector field $u\in H^{2}$ satisfying the absolute boundary
conditions \eqref{absolute-1}.
\end{lemma}

\begin{proof}
Apply Lemma \ref{lems2}(ii)
to a gradient field
$\nabla f$ for any scalar function $f$ on $\Omega$. Then
\begin{eqnarray}
\frac{1}{2}\partial_\nu(|\nabla f|^{2}) &=&-\pi \big((\nabla
f)^{\Vert}, (\nabla f)^{\Vert}\big) -H\left|\partial_\nu
f\right|^{2}+ \partial_\nu f \Delta f  \notag \\
&&+2\langle \left(\nabla f\right)^{\Vert },
  \nabla ^{\Gamma}(\partial_\nu f)\rangle
  -\nabla^{\Gamma}\cdot\big(\partial_\nu f \left(\nabla f\right)^{\Vert}\big).
\label{s29-71}
\end{eqnarray}
Write $u=(u^{1}, \cdots, u^{n})$ under an orthonormal basis of
$T\Omega $. Then, according to the Bochner identity \eqref{Bochner},
\begin{eqnarray*}
|\nabla^2 u|^{2}
=\sum_{k=1}^n|\nabla^2u^{k}|^{2} =\sum_{k=1}^n\big(\frac{1}{2}\Delta
|\nabla u^{k}|^{2}-\langle \nabla \Delta u^{k},\nabla u^{k}\rangle
\big),
\end{eqnarray*}
and, after integration, we find
\begin{eqnarray*}
\|\nabla^2 u\|_{2}^{2}
&=&\sum_{k+1}^n\left(\frac{1}{2}\int_{\Omega}\Delta |\nabla
u^{k}|^{2}\,
dx-\int_{\Omega}\langle \nabla\Delta u^{k},\nabla u^{k}\rangle \right)\, dx \\
&=&\sum_{k=1}^n\int_{\Omega}(\Delta u^{k})^{2}\, dx
+\frac{1}{2}\int_{\Gamma}\partial_\nu(|\nabla u^{k}|^{2})\,
dS-\int_{\Gamma
}(\Delta u^{k})\partial_\nu u^{k}\, dS \\
&=&\sum_{k=1}^n\int_{\Omega}(\Delta u^{k})^{2}\, dx
-\int_{\Gamma}\sum_{i,j=1}^{n-1}h_{ij}(\nabla_{i}u^{k})(\nabla_{j}u^{k})\, dS \\
&&-\int_{\Gamma}\sum_{k=1}^n |\partial_\nu u^{k}|^{2}H \, dS
+2\int_{\Gamma}\langle \nabla^{\Gamma}u^{k}, \nabla^{\Gamma}(\partial_\nu u^{k})\rangle\, dS \\
&=&-\int_{\Gamma}\sum_{i,j=1}^{n-1}h_{ij}\nabla_{i}u^{k}\nabla_{j}u^{k}
\, dS
-\int_{\Gamma}\sum_{k=1}^n\langle \nabla u^{k}, \mathbf{\nu}\rangle^{2}H\, dS \\
&&+2\sum_{k=1}^n\int_{\Gamma}\langle \nabla^{\Gamma}u^{k},
\nabla^{\Gamma}\langle \nabla u^{k}, \mathbf{\nu}\rangle \rangle \,
dS +\|\Delta u\|_{2}^{2},
\end{eqnarray*}
where the third equality follows from (\ref{s29-71}) applying to
each  $u^{k}$. We now handle the boundary integral:
\begin{eqnarray*}
\sum_{k=1}^n\int_{\Gamma}\langle \nabla^{\Gamma}u^{k},
\nabla^{\Gamma}(\partial_\nu u^{k})\rangle\, dS
&=&\sum_{k=1}^{n-1}\int_{\Gamma}\langle \nabla^{\Gamma}u^{k},
 \nabla^{\Gamma}(\partial_\nu u^{k})\rangle\, dS
+\int_{\Gamma}\langle \nabla^{\Gamma}u^{n},
\nabla^{\Gamma}(\partial_\nu u^{n})\rangle \, dS\\
&=&\sum_{k=1}^{n-1}\int_{\Gamma}\langle \nabla^{\Gamma}u^{k},
 \nabla^{\Gamma}(\partial_\nu u^{k})\rangle,
\end{eqnarray*}
where we have used the fact that $u^{n}|_{\Gamma}=0$ so that
$\nabla^{\Gamma}u^{n}=0$. For $1\le k\leq n-1$, we have
$\nabla_{n}u^{k}=\nabla_{k}u^{n}=0$ on $\Gamma$. However,
\begin{equation*}
\partial_\nu u^{k}=\nabla_{n}u^{k}-\sum_{j=1}^{n-1}u^{j}\Gamma_{nj}^{k}
=-\sum_{j=1}^{n-1}u^{j}\Gamma_{nj}^{k},
\end{equation*}
so that
\begin{eqnarray*}
\sum_{k=1}^{n-1}\langle\nabla^{\Gamma}u^{k},
\nabla^{\Gamma}(\partial_\nu u^{k})\rangle =-\sum_{k=1}^{n-1}\langle
\nabla^{\Gamma}u^{k},
\nabla^{\Gamma}\sum_{j=1}^{n-1}u^{j}\Gamma_{nj}^{k}\rangle\leq
C\left( |u|^{2}+|\nabla u|^{2}\right),
\end{eqnarray*}
and the inequality follows the trace imbedding theorem.
\end{proof}

\begin{corollary}\label{coro-s29-1}
Let $u\in H^{2}(\Omega)$ be a vector field which satisfies the
absolute boundary condition \eqref{absolute-1} on $\Gamma$. Let
$d=\nabla\cdot u$, $\omega=\nabla \times u$, and $\psi =\nabla
\times\omega$. Then

{\rm (i)} $\nabla\cdot\omega=\nabla\cdot\psi =0$;

{\rm (ii)} $\Delta u=\nabla d-\nabla \times \omega$;

{\rm (iii)} $\left. u^{\bot}\right|_{\Gamma}=0$, $\left.
\omega^{\Vert}\right|_{\Gamma}=0$, and
$\left.\psi^{\bot}\right|_{\Gamma}=0$;

{\rm (iv)} There exists $C>0$ depending only on $\Omega$ such that
\begin{equation*}
\|u\|_{H^{2}}\leq C \|(\nabla d, \nabla\times\omega, u)\|_{2}.
\end{equation*}
That is, $\|(\nabla d, \nabla\times \omega, u)\|_{2}$ is an
equivalent norm for any vector field $u\in H^{2}(\Omega)$ satisfying
the absolute boundary condition \eqref{absolute-1}.
\end{corollary}

Identity (i) is obvious, and (ii) follows from the vector identity:
\begin{equation*}
\Delta u=-\nabla \times \left( \nabla \times u\right) +\nabla
\left(\nabla\cdot u\right).
\end{equation*}
To show $\left. \psi^{\bot}\right|_{\Gamma}=0$,  we use (ii) to
obtain
\begin{equation*}
\left(\Delta u\right)^{\bot}=\left(\nabla d\right)^{\bot}
-\psi^{\bot},
\end{equation*}
and the claim follows from Lemma \ref{lems2}. According to Lemma
\ref{ths29-01} and the Ehrling-Nirenberg-Gagliardo interpolation
inequality,
one has
\begin{equation*}
\|u\|_{H^{2}}\leq C\big(\|\Delta u\|_{2}+\|u\|_{2}\big),
\end{equation*}
and (iv) follows from (ii).

\section{Navier-Stokes Equations}

Consider the Navier-Stokes equations \eqref{1.1}--\eqref{1.2}
in a bounded domain $\Omega\subset\mathbb{R}^{3}$,
subject to the kinematic condition \eqref{kinematic-1} and the
nonhomogeneous vorticity condition \eqref{vorticity-1}.

Using the vector identity:
\begin{equation}\label{4.2a}
u\cdot\nabla u=\omega \times u+\frac{1}{2}\nabla(|u|^{2}),
\end{equation}
one may rewrite (\ref{1.1}) as
\begin{equation}
\partial_t u+  \omega \times u=\mu
\Delta u-\nabla \big(p+\frac{1}{2}|u|^{2}\big).
\label{nsde8}
\end{equation}

\subsection{Neumann problem for the pressure $p$}

\begin{lemma}\label{les29-11}
If $u(t,x)$ is a smooth solution of the Navier-Stokes equations
\eqref{1.1}--\eqref{1.2} in $\Omega_T:=\Omega\times (0, T]$ subject
to the boundary conditions \eqref{kinematic-1}--\eqref{vorticity-1},
then $p$ solves the Neumann problem:
\begin{equation}\label{pressure-ns-1}
\Delta p=-\nabla\cdot\left(u\cdot\nabla u\right), \qquad \left.
\partial_\nu p\right|_{\Gamma}=\pi
(u,u)-\mu \nabla ^{\Gamma }\times a\text{.}
\end{equation}
\end{lemma}

\begin{proof}
Since $u(t,x)$ is a smooth solution in $\Omega_T$, by taking the
normal part of equations \eqref{1.1},
we may determine the normal derivative of the pressure $p$. Indeed,
since $\left. u^{\bot}\right|_{\Gamma}=0$, equations (\ref{1.1})
imply that
\begin{equation*}
\left( (u\cdot \nabla) u\right)^{\bot} =\mu\left(\Delta
u\right)^{\bot}-(\nabla p)^{\bot}\qquad\text{on the
boundary}\,\,\Gamma.
\end{equation*}
We apply Lemmas \ref{lems1}--\ref{lems2} to obtain
\begin{eqnarray*}
\left. \partial_\nu p\right|_{\Gamma}
&=&\mu\langle\Delta u,\, \mathbf{\nu}\rangle -\langle u\cdot\nabla u,\, \mathbf{\nu}%
\rangle  \notag \\
&=&\mu \partial_\nu(\nabla\cdot u)-\mu\langle\nabla\times\omega,\,
\mathbf{\nu}\rangle
+\pi(u,u)  \notag \\
&=&\pi(u,u)-\mu\nabla^{\Gamma}\times a.
\label{nssedq1}
\end{eqnarray*}
\end{proof}

\begin{remark} Similarly, if $u(t,x)$ is a smooth solution of the Euler equations
\eqref{1.4}--\eqref{1.5}
in $\Omega_T$ satisfying the kinematic condition
\eqref{kinematic-1},
then the pressure $p$ is determined by the Neumann problem of the
Poisson equation:
\begin{equation}
\Delta p=-\nabla\cdot\left(u\cdot\nabla u\right), \qquad \left.
\partial_\nu p\right|_{\Gamma}=\pi (u,u).
\label{pressure-euler-1}
\end{equation}
\end{remark}

\subsection{Reduction to an IBVP with absolute boundary conditions via
the Stokes equations}

We will employ the solution of the unsteady Stokes equations to
reduce our initial-boundary value problem into an initial-boundary
value problem with absolute boundary conditions. More precisely, let
$u$ be a solution of the initial-boundary value problem of the
Navier-Stokes equations (\ref{1.1}) with nonhomogeneous boundary
conditions \eqref{kinematic-1}--\eqref{vorticity-1}
and initial condition \eqref{1.3}.
Let $w$ be the solution of the nonhomogeneous initial-boundary
problem for the unsteady Stokes equations:
\begin{equation}
\left\{
\begin{array}{ll}
\partial_t w=\mu \Delta w-\nabla q, \qquad \nabla\cdot w=0,\\
\left. w^{\bot}\right|_{\Gamma}=0, \qquad\left. \left(\nabla\times
w\right)^{\Vert}\right|_{\Gamma }=a,\\
w|_{t=0}=u_{0}.
\end{array}
\right. \label{in-stokes}
\end{equation}
Then $v=u-w$ solves the following homogeneous initial-boundary value
problem:
\begin{eqnarray}
&&\partial_t v + \left(v+w\right)\cdot\nabla \left( v+w\right)
=\mu \Delta v-\nabla p,\label{ns-2a}\\
&& \nabla\cdot v=0, \label{ns-2b}\\
&&v|_{t=0}=0,\label{ns-2c}
\end{eqnarray}
subject to the absolute boundary condition:
\begin{equation}\label{abs-2}
v^{\bot}=0, \qquad (\nabla\times v)^{\Vert}=0.
\end{equation}

Therefore, if the initial-boundary value problem (\ref{in-stokes})
is well-posed with some appropriate estimates of its solution, then
the well-posedness for the nonhomogeneous initial-boundary value
problem \eqref{1.1}--\eqref{1.3} and
\eqref{kinematic-1}--\eqref{vorticity-1} is equivalent to the
well-posedness of the homogeneous initial-boundary value problem
\eqref{ns-2a}--\eqref{abs-2}.

\subsection{Well-posedness of the Stokes equations and $L^2$-estimates for the
solutions}

The initial-boundary value problem (\ref{in-stokes}) can be solved
as follows. Taking the divergence in (\ref{in-stokes}), we can
decouple the scaler function $q$ by solving the Neumann boundary
problem:
\begin{equation}
\Delta q=0, \qquad \partial_\nu q|_{\Gamma}=-\mu \nabla^{\Gamma
}\times a, \label{decop-par}
\end{equation}
where the normal derivative of $q$ follows from the kinematic
condition in \eqref{in-stokes} and $\nabla\cdot w=0$, which implies
from Lemma \ref{lems2} that
\begin{eqnarray*}
\partial_\nu q|_{\Gamma}=\mu \langle \Delta u, \nu\rangle
=-\mu \nabla^{\Gamma}\times \left(\nabla\times w\right)^{\Vert}
=-\mu \nabla^{\Gamma }\times a.
\end{eqnarray*}
Since $\int_{\Gamma}\nabla^{\Gamma}\times a\, dS=0$, there exists a
unique solution $q$ modulo a constant. We renormalize $q$ so that
$\int_{\Omega} q\,dx =0$. Therefore, it suffices to solve the linear
parabolic equations:
\begin{equation}
\left\{
\begin{array}{ll}
\partial_t w-\mu \Delta w=-\nabla q,\\
\left. w^{\bot}\right|_{\Gamma}=0, \qquad\left. \left( \nabla
\times w\right)^{\Vert }\right|_{\Gamma }=a,\\
w|_{t=0}=u_{0},
\end{array}
\right.   \label{linear-p}
\end{equation}
where $q$ is given by (\ref{decop-par}), which is well-posed.
In fact,
by taking any smooth vector field $A$ such that $\left. A^{\bot
}\right|_{\Gamma }=0$ and $\left. (\nabla \times A)^{\Vert
}\right|_{\Gamma }=a$, and setting $v:=w-A$, then the linear
parabolic problem can be written as
\begin{equation*}
\left\{
\begin{array}{ll}
\frac{\partial }{\partial t}v=\mu \Delta v+f\,,
\\
v^{\bot }=0, \qquad\left(\nabla \times v\right)^{\Vert}=0,
\end{array}
\right.
\end{equation*}
where $f:=\mu\Delta A-\nabla q-\frac{\partial }{\partial t}A$, which
is well-posed; its proof directly follows the arguments in Section
5. Furthermore, all the required estimates on $w$ and $q$ are
available.

\bigskip
We now make some necessary estimates for the solution $w$
of \eqref{in-stokes}.

Taking the curl operation $%
\nabla \times$ in the equations in \eqref{linear-p}, with $g=\nabla
\times w$ and $h =\nabla \times g=-\Delta w $,
\begin{equation*}
\partial_t g=\mu \Delta g; \qquad \partial_t h =\mu \Delta h,
\end{equation*}
together with the boundary conditions:
\begin{equation*}
g^{\Vert}|_{\Gamma}=a,  \qquad\left(\nabla\times
g\right)^{\bot}|_{\Gamma}=\nabla^{\Gamma}\times a,
\end{equation*}
and
\begin{equation*}
h^{\bot}|_{\Gamma}=\nabla^{\Gamma}\times a, \qquad
\left(\nabla\times h \right)^{\Vert}|_{\Gamma}=-\frac{1}{\mu
}\partial_t a.
\end{equation*}

We now make the $L^{2}$-estimates for $g$ and $h$. First note the
following identities:
\begin{equation*}
\partial_t(|g|^{2})=2\mu \langle g,\Delta g\rangle, \qquad
\partial_t(|h|^{2})=2\mu \langle h, \Delta h
\rangle.
\end{equation*}
We integrate over $\Omega$ to obtain
\begin{eqnarray*}
\frac{d}{dt}\|g\|_{2}^{2}&=&-2\mu \int_{\Omega }\langle
g,\nabla \times (\nabla \times g)\rangle\,dx \\
&=&-2\mu \int_{\Omega }\langle \nabla \times g,\nabla \times
g\rangle\,dx +2\mu \int_{\Gamma }\langle (\nabla \times g)^{\Vert
}\times a,\nu \rangle\, dS,
\end{eqnarray*}
which implies
\begin{equation}\label{4.12a}
\|g\|_2^2(t)+2\mu\int_0^t\|\nabla\times g\|_2^2(s)\, ds
=\|g\|_2^2(0)+ 2\mu\int_0^t\int_{\Gamma}\langle (\nabla\times
g)^{\Vert}\times a,\, \nu \rangle\, dS ds \qquad\text{for}\,\,
0<t\le T.
\end{equation}
Similarly, we have
\begin{eqnarray*}
\frac{d}{dt}\|h\|_{2}^{2} &=&-2\mu \int_{\Omega }\langle
h, \nabla \times (\nabla \times h)\rangle\,dx \\
&=&-2\mu \int_{\Omega }\langle \nabla \times h, \nabla \times
h\rangle\,dx -2\int_{\Gamma }\langle (\partial_t a)\times h, \nu
\rangle\, dS,
\end{eqnarray*}
which implies
\begin{equation}\label{4.13a}
\|h\|_2^2(t)+2\mu\int_0^t\|\nabla\times h\|_2^2(s)\,ds
=\|h\|_2^2(0)- 2\int_0^t\int_{\Gamma}\langle (\partial_t a)\times
h,\, \nu \rangle\, dSds \qquad\text{for}\,\, 0<t\le T.
\end{equation}

We now consider two different cases.

\bigskip
{\bf Case 1: $\partial_t a=0$ and $\sqrt{\mu}a\in L^2(\Gamma_T)$}.
In this case, we obtain from \eqref{4.13a} that
$$
\|h\|_2^2(t)+2\mu\int_0^t\|\nabla\times h\|_2^2(s)\,ds
=\|h\|_2^2(0)\le \|w(0,\cdot)\|_{H^2} =\|u_0\|_{H^2}.
$$
Since $\nabla\cdot g=\nabla\cdot h=0$, we conclude
\begin{equation}\label{4.14a}
\|\nabla g\|_{L^2(\Omega)}^2(t)+\mu\|\nabla^2g\|_{L^2(\Omega_T)}^2
\le C \qquad \text{for}\,\, t\in [0, T],
\end{equation}
where $C>0$ is independent of $\mu$.

Then we have
\begin{eqnarray*}
&&2\mu \Big|\int_0^t\int_\Gamma\langle (\nabla\times
g)^{\Vert}\times a,
\nu\rangle\, dS ds\Big|\\
&&\le 2\mu \Big(\varepsilon \int_0^t\int_\Gamma |\nabla g|^2\, dSds
     + C_\varepsilon\|a\|_{L^2(\Gamma_T)}^2\Big)\\
&&\le \mu\int_0^t\int_\Omega|\nabla\times g|^2dx
   +C\mu \|\nabla^2
g\|_{L^2(\Omega_T)}^2+\mu\|a\|_{L^2(\Omega_T)}^2\\
&&\le \mu\int_0^t\int_\Omega|\nabla\times g|^2dx
  +C\|u_0\|_{H^2(\Omega)}^2+\mu\|a\|_{L^2(\Gamma_T)}^2,
\end{eqnarray*}
where we have used the interpolation inequality in the second
inequality above.

Substitution this into \eqref{4.12a} yields
\begin{equation}\label{4.14b}
\|g\|_{L^2(\Omega)}^2(t)+\mu \|\nabla g\|_{L^2(\Omega_t)}^2 \le
C\|u_0\|_{H^2(\Omega)}^2+\mu \|a\|_{L^2(\Omega_T)}^2.
\end{equation}

Combining \eqref{4.14a}--\eqref{4.14b} with the fact $\nabla\cdot
w=\nabla\cdot g=\nabla\cdot h=0$, we have
\begin{equation}\label{4.15a}
\|w\|_{H^2(\Omega)}^2(t)+\mu \|w\|_{H^3(\Omega_T)}^2 \le
C\|u_0\|_{H^2(\Omega)}^2+\mu
\|a\|_{L^2(\Omega_T)}^2\qquad\text{for}\,\, t\in [0,T],
\end{equation}
where $C>0$ is independent of $\mu$.

\bigskip
{\bf Case 2: $\partial_t a\ne 0$ and $(a, \partial_ta)\in
L^2(\Gamma_T)$}. In this case, we have
\begin{eqnarray*}
&&2\Big|\int_0^t\int_\Gamma\langle (\partial_t a)\times h,
\nu\rangle\, dS ds\Big|\\
&&\le \varepsilon \int_0^t\int_\Gamma |h|^2\, dSds
     + C_\varepsilon\|\partial_ta\|_{L^2(\Gamma_T)}^2\\
&&\le \mu\int_0^t\int_\Omega|\nabla\times h|^2dx ds
   +\frac{C}{\mu}\int_0^t\int_\Omega |h|^2dxds+C\|\partial_t a\|_{L^2(\Omega_T)}^2,
\end{eqnarray*}
where we have used $\nabla\cdot h=0$ and the interpolation
inequality in the second inequality above.

Substitution this into \eqref{4.13a} yields
$$
\|h\|_2^2(t)+\mu\int_0^t\|\nabla\times h\|_2^2(s)ds \le
\|u_0\|_{H^2(\Omega)}^2 +C\|\partial_ta\|_{L^2(\Gamma_T)}
+\frac{C_0}{\mu}\int_0^t\|h\|_2^2(s)\,ds.
$$
Then the Gronwall inequality yields
$$
\|h\|_2^2(t) \le
M(\mu,T)(\|u_0\|_{H^2(\Omega)}^2+\|\partial_ta\|_{L^2(\Gamma_T)}^2).
$$
Combining this with \eqref{4.12a}, we conclude
\begin{equation}\label{4.16}
\|w\|_{H^2(\Omega)}(t) + \|w\|_{H^3(\Omega_T)}^2 \le
M(\mu,T)(\|u_0\|_{H^2(\Omega)}^2+\|(a,
\partial_ta)\|_{L^2(\Gamma_T)}^2),
\end{equation}
where $M(\mu,T)>0$ depends only on $\mu$ and $T$.

\bigskip
\begin{proposition} \label{prop:4.1}
Let $w\in H^3(\Omega)$ be a solution of problem \eqref{in-stokes}.
Then

{\rm (i)} If $\partial_t a=0$ and $\sqrt{\mu}a\in L^2(\Gamma_T)$,
then there exists $C>0$, independent of $\mu$, such that
\eqref{4.15a} holds;

{\rm (ii)} If $\partial_t a\ne 0$ and $(a, \partial_ta)\in
L^2(\Gamma_T)$, there exists $M=M(\mu,T)>0$ such that \eqref{4.16}
holds.
\end{proposition}

\section{A linear initial-boundary value problem}

In this section, we deal with the absolute boundary problem of a
linearized parabolic system.
Let $\beta(t,x), 0\le t\leq T$, be a smooth vector field on $\Omega$
satisfying the absolute boundary condition:
$$
\left. \beta ^{\bot}\right|_{\Gamma}=\left.
(\nabla\times\beta)^{\Vert}\right|_{\Gamma}=0.
$$
Define the pressure function $p_{\beta}(t,x)$ for $0\le t\leq T$ by
solving the Poisson equation:
\begin{equation}
\left\{
\begin{array}{ll}
\Delta p_{\beta}=-\nabla\cdot\big((\beta+w)\cdot\nabla(\beta+w)\big), \\
\left. \partial_\nu p_{\beta}\right|_{\Gamma}=\pi \left(\beta +w,
\beta +w\right),
\end{array}
\right.  \label{ps-01}
\end{equation}
subject to the normalization $\int_{\Omega}p_{\beta}\, dx =0$. Here
and hereafter, the time-variable $t$ is omitted for simplicity in
the equations if no confusion may arise. Thanks to Lemma
\ref{lems1},
\begin{eqnarray*}
\int_{\Omega}\nabla\cdot\big((\beta +w)\cdot\nabla(\beta
+w)\big)\,dx =\int_{\Gamma}\langle (\beta +w)\cdot\nabla (\beta +w),
\mathbf{\nu}\rangle\,dS  =-\int_{\Gamma}\pi \left(\beta +w, \beta
+w\right)\, dS,
\end{eqnarray*}
hence the above boundary value problem (\ref{ps-01}) has a unique
solution. The elliptic estimates yield that $p_{\beta }$ is smooth
both in $t$ and $x$.

\medskip
We now consider the following initial-boundary value problem in
$\Omega$:
\begin{equation}
\partial_t v+(\beta+w)\cdot \nabla (v+w)=\Delta v-\nabla
p_{\beta},\label{ps-02}
\end{equation}
subject to the initial-boundary conditions:
\begin{eqnarray}
&&v|_{t=0}=0, \label{ps-in-1}\\
&&\left. v^{\bot}\right|_{\Gamma}=\left.
\omega^\Vert\right|_{\Gamma}=0. \label{ps-bd-1}
\end{eqnarray}
In this section and Section 6, we denote $\omega=\nabla\times v$
again without confusion and always set $\mu=1$ without loss of
generality for the existence proof.

\medskip
To solve problem \eqref{ps-02}--\eqref{ps-bd-1}, we recall that the
Laplacian $\Delta$ acting on vector fields in
\begin{equation*}
D(\Delta)=\big\{ v\in H^{2}(\Omega;\R^3)\, :\,
v^{\bot}\big|_{\Gamma}=\omega^{\Vert}\big|_{\Gamma}=0\big\}
\end{equation*}
is self-adjoint on $L^{2}(\Omega)$ and is negative-definite:
\begin{eqnarray*}
(\Delta v, v)=-\int_{\Omega}|\nabla\cdot
v|^{2}\,dx-\int_{\Omega}|\omega|^{2}\,dx \leq 0
\end{eqnarray*}
for every $v\in D(\Delta)$. The Laplacian with domain $D(\Delta)$
will be still denoted by $\Delta$ for simplicity.
According to Hille-Yosida's theorem, $\Delta$ is the infinitesimal
generator of a strongly continuous semigroup of contractions on
$L^{2}(\Omega)$, denoted by $\left( e^{-t\Delta}\right)_{t\geq 0}$.
Indeed, $\Delta$ is the infinitesimal generator of an analytic
semigroup, and the $L^{2}$-domain $D(\Delta)$ of $\Delta$ is
invariant under $e^{-t\Delta}$. In addition, for each $t$,
$e^{-t\Delta}$ commutes with the curl operation $\nabla \times $ and
the divergence operator $\nabla\cdot$ (see \cite{conner1956} for the
details).

For every $f\in D(\Delta)$, $\varphi(t)=e^{-t\Delta}f$ solves the
following evolution equation:
\begin{equation*}
\partial_t\varphi=\Delta\varphi, \qquad \varphi(0,\cdot)=f.
\end{equation*}
While, the regularity theory for parabolic equations yields that, if
$f$ is smooth, so is $\varphi$.

\medskip
Our next aim is to solve the homogenous equation:
\begin{equation*}
\partial_t\varphi=\big(\Delta-\theta\cdot\nabla\big)\varphi,\qquad
\varphi|_{t=0}=f,
\end{equation*}
where $f\in D(\Delta)$, and $\theta$ is a smooth vector field on
$\Omega$ (independent of $t$).
To this end, we show that $A=\Delta-\theta\cdot\nabla$ with domain
$D(A)=D(\Delta)$ is the generator of a strongly-continuous
semigroup.

\begin{lemma}\label{lems4} Under the above notations, $(A,
D(\Delta))$ is a densely defined, closed operator, which is indeed
the infinitesimal generator of an analytic semigroup on
$L^{2}(\Omega)$.
\end{lemma}

\begin{proof}
If a vector field $\varphi$ satisfies the absolute boundary
conditions, then
\begin{eqnarray*}
\int_\Omega|\nabla\varphi|^2\,dx
&=&\frac{1}{2}\int_\Omega\Delta(|\varphi|^2)\, dx-\int_\Omega\langle
\Delta\varphi,\varphi\rangle\, dx\\
&=&\frac{1}{2}\int_\Gamma\partial_\nu(|\varphi|^2)\,
dS-\int_\Omega\langle\Delta\varphi,
\varphi\rangle\, dx\\
&=&-\int_\Gamma\pi(\varphi,
\varphi)\, dS-\int_\Omega\langle\Delta\varphi, \varphi\rangle\rangle\, dx\\
&\le&\int_\Gamma |\varphi|^2\, dS-\int_\Omega \langle<\Delta\varphi,
\varphi\rangle\, dx\\
&\le& \varepsilon
\|\Delta\varphi\|_2^2+C\|\varphi\|_2^2+\|\Delta\varphi\|_2\|\varphi\|_2,
\end{eqnarray*}
that is,
$$
\|\nabla\varphi\|_2^2\le
C\big(\|\varphi\|_2^2+\|\Delta\varphi\|_2\|\varphi\|_2\big).
$$
Hence, we have
\begin{eqnarray*}
\|(\theta\cdot\nabla)\varphi\|^{2} \leq \|\theta\|_{\infty}^{2}
\|\nabla \varphi\|^{2}\leq C\|\theta\|_{\infty}^{2}
\big(\|\varphi\|_2^2+\|\Delta\varphi\|_2\|\varphi\|_2\big)
\le C \|\theta\|_\infty^2\|\varphi\|^2_2+\varepsilon
\|\Delta\varphi\|_2^2,
\end{eqnarray*}
for any $\varepsilon \in (0,1)$, where $C(\varepsilon,
\|\theta\|_{\infty}^{2})$ is a constant depending only on
$\varepsilon$ and $\|\theta\|_{\infty}^{2}$. According to Kato's
perturbation theorem (for example, see Theorem 2.1, page 80 in \cite
{pazy1983}), all the conclusions follow.
\end{proof}

The analytic semigroup with infinitesimal generator
$\Delta-\theta\cdot\nabla$ with domain $D(\Delta)$ is denoted by
$\big(e^{t(\Delta_{M}-\theta\cdot\nabla)}\big)_{t\geq 0}$.

\medskip
Next, we want to solve
the nonhomogeneous evolution equation
\begin{equation*}
\partial_t\varphi=\big(\Delta-(\beta+w)\cdot\nabla \big)\varphi
 +h, \qquad \varphi(0,\cdot)=f,
\end{equation*}
where $\beta(t,x)$ is a smooth vector field and satisfies the
absolute boundary conditions such that
\begin{equation*}
\|\beta(t)-\beta(s)\|_{\infty}\leq C|t-s|^{\alpha}
\end{equation*}
for some constant $C$ and $\alpha \in (0,1]$, $w(t,x)$ is the unique
solution of \eqref{in-stokes}, $h(t, \cdot)\in L^{2}(\Omega)$ for
each $t\in [0,T]$, and $f\in D(\Delta)$.

\begin{lemma}\label{lemrg01}
Let
\begin{equation*}
\rho=\sup_{0\le t\leq T}\left\{\int_{\Omega}\langle\big(\Delta
-(\beta+w)(t,\cdot)\cdot\nabla\big)\psi, \psi\rangle\, :\, \psi\in
D(\Delta_{M}) \text{ and } \|\psi\|=1\right\}.
\end{equation*}
Then $\rho <\infty $. If $\psi\in D(\Delta)$ satisfies the Poisson
equation:
\begin{equation*}
\big(\Delta -(\beta+w)(t,\cdot)\cdot\nabla -\rho-1\big)\psi=\varphi,
\qquad t\leq T,
\end{equation*}
and $\varphi\in L^{2}(\Omega)$, then $\|\psi\|\leq \|\varphi\|$ and
$\|\nabla \psi\|\leq C\|\varphi\|$ for some constant $C$.
\end{lemma}

\begin{proof} It is easy to see that
\begin{equation*}
\int_{\Omega}\langle\big(\Delta-(\beta+w)(t,\cdot)\cdot\nabla\big)\psi,
\psi\rangle\,dx \leq
\|(\beta+w)(t,\cdot)\|_{\infty}\int_{\Omega}|\psi|^{2}\,dx
\end{equation*}
so that $\rho <\infty$. Since
\begin{eqnarray*}
-\int_{\Omega}\langle \varphi, \psi\rangle\,dx =\int_{\Omega}\langle
\big(-\Delta +(\beta+w)(t,\cdot)\cdot\nabla +\rho +1\big)\psi,
\psi\rangle\, dx \geq \|\psi\|^2,
\end{eqnarray*}
which gives the first estimate in the lemma. Using the Bochner
identity \eqref{Bochner}, we have
\begin{eqnarray*}
\int_\Omega|\nabla\psi|^2\,
dx&=&\frac{1}{2}\int_\Omega\Delta(|\psi|^2)\,dx-\int_\Omega\langle
\Delta\psi, \psi\rangle\,dx
=\frac{1}{2}\int_\Gamma\partial_\nu(|\psi|^2)\,
dS-\int_\Omega\langle\Delta\psi,
\psi\rangle\,dx\\
&=&-\int_\Gamma\pi(\psi, \psi)\,dS-\int_\Omega\langle\Delta\psi,
\psi\rangle\,dx\\
&=&\int_\Gamma|\psi|^2\,dS -\int_\Omega\langle\varphi
+\big((\beta+w)(t,\cdot)\cdot\nabla+\rho+1\big)\psi,
\psi\rangle\,dx\\
&\le& \varepsilon\|\nabla\psi\|_2^2+C\|\varphi\|^2_2
 +\|\varphi\|_2\|\psi\|_2+\|(\beta+w)(t,\cdot)\|_\infty\|\nabla\psi\|_2\|\psi\|_2
 +(\rho+1)\|\psi\|_2^2,
\end{eqnarray*}
which implies
\begin{equation*}
\|\nabla \psi\|\leq C\|\varphi\|.
\end{equation*}
\end{proof}

\begin{lemma}\label{lems6}
Let $\beta(t,x), 0\le t\leq T,$ be a smooth vector field on $\Omega$
satisfying the absolute boundary conditions \eqref{absolute-1}.
Suppose that
\begin{equation*}
\|(\beta+w)(t,\cdot)\|_{\infty}\leq \rho \qquad \text{for every }%
t\in \lbrack 0,T]
\end{equation*}
for some non-negative constant $\rho $. Let
\begin{equation*}
A(t)=\Delta -(\beta+w)(t,\cdot)\cdot\nabla -(\rho +1)I
\end{equation*}
with domain $D(A(t))=D(\Delta)$. Then, for each $t$, $A(t)$ is the
infinitesimal generator of the strongly continuous semigroup
$\{e^{-(\rho
+1)s}e^{s(\Delta-(\beta+w)(t,\cdot)\cdot\nabla)}\}_{s\geq 0}$ of
contractions on $L^{2}(\Omega)$. Moreover, for every $s,t$, and
$\tau$ in $[0,T]$, we have
\begin{equation*}
\|\big(A(t)-A(s)\big)A(\tau )^{-1}\|\leq C\|(\beta+w)(t,\cdot
)-(\beta+w)(s,\cdot)\|_{\infty}
\end{equation*}
for some constant $C$.
\end{lemma}

\begin{proof}
Let $\psi=A(\tau)^{-1}\varphi$. Then $\psi\in D(\Delta)$ and,
according to Lemma \ref{lemrg01}, $\|\nabla\psi\|\leq C\|\varphi\|$
for some constant $C$. Since
\begin{equation*}
\big(A(t)-A(s)\big)A(\tau)^{-1}\varphi
=\big((\beta+w)(t,\cdot)-(\beta+w)(s,\cdot)\big)\cdot\nabla
\psi
\end{equation*}
so that
\begin{eqnarray*}
\|\big(A(t)-A(s)\big)A(\tau)^{-1}\varphi\| &\leq&
\|(\beta+w)(t,\cdot)-(\beta+w)(s,\cdot)\|_{\infty}\|\nabla\psi\|\\
&\leq&
C\|(\beta+w)(t,\cdot)-(\beta+w)(s,\cdot)\|_{\infty}\|\varphi\|.
\end{eqnarray*}
\end{proof}

Let us retain the notations as above. Assume that
\begin{equation*}
\|(\beta+w)(t,\cdot)-(\beta+w)(s,\cdot)\|_{\infty} \leq
C|t-s|^{\alpha}\qquad \text{for all}\,\, s,t\in \lbrack 0,T].
\end{equation*}
We want to solve the nonhomogeneous evolution equation:
\begin{equation*}
\partial_t \psi(t,\cdot)+ (\beta+w)(t,\cdot)\cdot\nabla \psi(t,\cdot)
=\Delta \psi(t,\cdot)+h(t),
\end{equation*}
where $\psi(t,\cdot)$ satisfies the absolute boundary conditions
\eqref{absolute-1}. The above evolution equation may be rewritten as
\begin{equation*}
\partial_t\Psi(t)=A(t)\Psi(t)+e^{-\left(\rho +1\right)t}h(t),
\end{equation*}
where $A(t)$ is given in Lemma \ref{lems6}, $\psi(t)=e^{(\rho
+1)t}\Psi(t)$, and $\Psi (t)$ satisfies the absolute boundary
condition. According to Theorem 6.1 in \cite{pazy1983}, page 150,
there is a unique evolution system $U(t,s)$ on $0\leq s\leq t\leq T$
such that the mild solution $\Psi(t)$ is given by
\begin{equation*}
\Psi(t)=U(t,0)\psi(0,\cdot)+\int_{0}^{t}e^{-(\rho+1)s}U(t,s)h(s,\cdot)ds
\end{equation*}
so that
\begin{eqnarray*}
\psi(t,\cdot) =e^{(\rho +1)t}U(t,0)\psi(0,\cdot)
+e^{(\rho+1)t}\int_{0}^{t}e^{-(\rho+1)s}U(t,s)h(s,\cdot)ds.
\end{eqnarray*}
Therefore, we have the following theorem.

\begin{theorem}\label{ths1}
Given the initial vector field $f\in D(\Delta)$, $h(t,\cdot)\in
L^{2}(\Omega)$, and $\beta(t,\cdot)$ for $t\in [0, T]$ such that
\begin{equation*}
\left\{\begin{array}{ll}
\beta(t,\cdot)\in D(\Delta_{M}) \quad
\text{for each $t\in [0,
T]$},\\
\sup_{0\le t\le T}\|\beta(t,\cdot)\|_{\infty }<\infty,\\
\|(\beta+w)(t,\cdot)-(\beta+w)(s,\cdot)\|_{\infty}\leq
C|t-s|^{\alpha}\qquad \forall s,t\in \lbrack 0,T],
\end{array}\right.
\end{equation*}
there exists a unique solution $\psi$ which solves the absolute
boundary problem on $\Omega$ of the linear parabolic equation:
\begin{equation*}
\partial_t\psi+(\beta+w)\cdot\nabla\psi=\Delta \psi+h.
\end{equation*}
\end{theorem}

\section{Construction of local solutions}

The main goal of this section is to prove the existence of a unique
strong solution, local in time, to the Navier-Stokes equations with
the nonhomogeneous boundary conditions. To this end, we establish
apriori estimate for solutions of the linear parabolic equations
(\ref{ps-02}).

In order to state our results, we use the following norm: For a
vector field $v(t, x), 0\le t\leq T, x\in\Omega$,
\begin{equation*}
\|v(t,\cdot)\|_{\mathbf{N}}=\sqrt{\|v(t,\cdot)\|_{H^2}^{2}+
\|v_{t}(t,\cdot)\|_{H^1}^{2}}.
\end{equation*}

\subsection{Main estimate}

Let $T>0$ be a fixed but arbitrary constant. Let $\beta(t,x), 0\le
t\leq T,$ be a given smooth vector field satisfying the absolute
boundary conditions \eqref{absolute-1}. Let $v=V(\beta)$ is the
unique solution to the linear parabolic equations
\eqref{ps-02}--\eqref{ps-bd-1},
with $p_{\beta}$ is the unique solution to problem \eqref{ps-01}
and $w(t,x)$ the unique solution of the initial-boundary value
problem \eqref{in-stokes} for the Stokes equations. Then the main
\emph{apriori} estimate is given in the following theorem.

\begin{theorem}
\label{them01} Let $\beta(t, x), 0\le t\leq T,$ be a vector field
which satisfies the following conditions:

{\rm (i)} For every $0\le t\leq T$, $\beta (t,\cdot )\in
H^2(\Omega)$ and $\partial_t\beta (t,\cdot)\in H^{1}(\Omega)$;

{\rm (ii)} For each $0\le t\leq T$, $\left.
\beta^{\bot}\right|_{\Gamma}=\left. (\nabla \times
\beta)^{\Vert}\right|_{\Gamma}=0$;

{\rm (iii)} $t\rightarrow
\|\beta(t,\cdot)\|_{H^{2}}^{2}+\|\partial_t\beta
(t,\cdot)\|_{H^{1}}^{2}$ is continuous;

{\rm (iv)} $\beta(0,\cdot)=0$.

Let $v=V(\beta)$ defined by \eqref{ps-01}--\eqref{ps-bd-1}.
Then there exists a constant $C$ depending only on the domain
$\Omega$ such that the following inequality holds:
\begin{eqnarray}
\|v(t,\cdot)\|_{\mathbf{N}}^{2} \leq C\|u_{0}\|_{H^{2}}^{4}e^{CQ(t)}
+C\int_{0}^{t}e^{C(Q(t)-Q(s))}\big(1+\|(\beta,
w)(s,\cdot)\|_{\mathbf{N}}^{2} \big)^2ds, \label{s29-071}
\end{eqnarray}
where
\begin{equation}
Q(t)=\int_{0}^{t}\big(1+
\|(\beta,w)(s,\cdot)\|_{\mathbf{N}}^{2}\big)ds. \label{s03-01}
\end{equation}
\end{theorem}

To establish estimate (\ref{s29-071}), we will frequently use an
elementary $L^{2}$-estimate for elliptic equations, which can be
stated as follows: If $\phi$ is a solution to the Neumann boundary
problem:
\begin{equation}
\left\{
\begin{array}{ll}
\Delta \phi =\nabla\cdot f, \\
\left. \partial_\nu \phi\right|_{\Gamma}=\langle f, \mathbf{\nu
}\rangle,
\end{array}
\right. \text{\ \ }  \label{posest01}
\end{equation}
then we have the Solonnikov estimate:
\begin{equation}
\|\nabla \phi\|_{2}\leq \|f\|_{2}. \label{es29-04}
\end{equation}

Estimate (\ref{es29-04}) is a special case of Theorem 2.2 in
Solonnikov \cite{Solonnikov1973}. In fact, estimate (\ref{es29-04})
follows easily from an integration by parts argument. Since
\begin{eqnarray*}
\|\nabla \phi\|_{2}^{2}&=&-\int_{\Omega}\langle \phi, \Delta \phi
\rangle \, dx +\int_{\Gamma}\phi\, \partial_\nu \phi \,dS
=-\int_{\Omega}\langle \phi, \nabla\cdot w\rangle\,dx
+\int_{\Gamma}\phi\, \partial_\nu\phi\, dS \\
&=&\int_{\Omega}\langle \nabla \phi, w\rangle\, dx -\int_{\Gamma}
\phi \langle w, \mathbf{\nu}\rangle\, dS +\int_{\Gamma}\phi
\,\partial_\nu \phi\, dS =\int_{\Omega}\langle \nabla\phi,
w\rangle\,dS,
\end{eqnarray*}
so that (\ref{es29-04}) follows from the Schwartz inequality directly.

\medskip
We now present the proof of Theorem \ref{them01}. Throughout the proof, we
will use $C$
to denote a constant depending only on the domain $\Omega$,
which may be different at each
occurrence. We also omit the lowerscript $\beta$ for simplicity: For
example, we write $p$ for $p_{\beta}$.

Let $d=\nabla\cdot v$, $\omega =\nabla\times v$, and $\psi =\nabla
\times \omega$ for simplicity. Let $f_{t}$ denote the
time-derivative $\partial_t f$ for a vector field $f$.
Observe that $v_{t}$ again satisfies the absolute boundary
conditions \eqref{absolute-1}. Thus, according to Corollary
\ref{coro-s29-1},
\begin{eqnarray}
\|v_{t}\|_{H^{1}}\leq C \|(v_{t}, d_{t},\omega_{t})\|_{2}, \qquad
&&\|v\|_{H^{2}}\leq C\|(v, \nabla d, \psi)\|_{2}. \label{s29-83}
\end{eqnarray}
It is therefore natural to bound the time-derivative of each term
appearing on the right-hand sides of (\ref{s29-83}).
However, as a matter of fact, we are unable to
bound $\frac{d}{dt}\|\nabla d\|_{2}^{2}$. A different approach to
handle $\|\nabla d\|_{2}$ is required.

Taking the divergence $\nabla\cdot$ on both sides of the linear
parabolic equations (\ref{ps-02}), one obtains
\begin{equation}
\partial_t d=\Delta d
+\nabla\cdot\left((\beta+w)\cdot\nabla(\beta-v)\right).
\label{t29-21}
\end{equation}
{}From \eqref{ps-02}--\eqref{ps-bd-1},
we have
\begin{eqnarray*}
\langle \Delta v, \,\nu \rangle =-\pi (\beta +w, v+w)+
\partial_\nu\, p_{\beta} =\pi\left(\beta+w, \beta -v\right)
\end{eqnarray*}
so that, according to Lemma \ref{lems2}, $d$ satisfies the Neumann
boundary condition:
\begin{equation}
\left. \partial_\nu d \right|_{\Gamma}=\pi\left(\beta+w,
\beta-v\right),
\label{t29-22}
\end{equation}
which is nonhomogeneous. We define, for every $0\le t\leq T$, a function $%
q(t,\cdot)$ by solving the Poisson equation:
\begin{equation}
\left\{
\begin{array}{ll}
\Delta q=-\nabla\cdot\left((\beta+w)\cdot\nabla(\beta-v)\right), \\
\left. \partial_\nu q\right|_{\Gamma} =\pi\left(\beta+w,
\beta-v\right),
\end{array}
\right.
\label{sa29-01}
\end{equation}
subject to \ $\int_{\Omega}q \, dx=0$.

Let $g=d-q$. Then $g$ solves the linear parabolic equation:
\begin{equation}
\left\{
\begin{array}{ll}
\partial_t g=\Delta g-q_{t}, \\
\left. \partial_\nu g\right|_{\Gamma}=0.
\end{array}
\right.
\label{sa29-04}
\end{equation}
Since $\|\nabla d\|_{2}\leq \|(\nabla g, \nabla q)\|_{2}$, we have
\begin{equation*}
\|v\|_{H^{2}}\leq C \|(v, \nabla g, \nabla q, \psi)\|_{2}.
\end{equation*}
We thus consider the following function:
\begin{eqnarray*}
F(t)=\|(v,\psi)(t,\cdot)\|_{2}^{2}+\|(\nabla g, \nabla
q)(t,\cdot)\|_{2}^{2}
+\|(v_{t}, d_t, \omega_t)(t,\cdot)\|_{2}^{2}.
\end{eqnarray*}
Then $\|v(t,\cdot)\|_{\mathbf{N}}^{2}\leq F(t)$, and
\begin{equation*}
F(0)=\|(u_{0}\cdot\nabla)u_{0}+\nabla p_{0}\|_{2}^{2}
+\|\nabla\times(u_{0}\cdot\nabla)u_{0}\|_{2}^{2},
\end{equation*}
where $p_{0}$ is the solution to
\begin{equation*}
\Delta p_{0}=-\nabla\cdot\left((u_{0}\cdot\nabla)u_{0}\right),
\qquad \left. \partial_\nu p_{0}\right|_{\Gamma} =\pi(u_{0}, u_{0}).
\end{equation*}
Hence
\begin{equation*}
F(0)\leq C\|u_{0}\|_{H^{2}}^4\text{,}
\end{equation*}
and (\ref{s29-071}) follows from the following lemma.

\begin{lemma}
Under the assumptions in Theorem {\rm \ref{them01}} and notations
introduced above, we have the following differential inequality:
\begin{equation}
\frac{d}{dt}F\leq C\left(\|(\beta, w)\|_{\mathbf{N}}^{2}+1\right)
\left(F+ 1+\|(\beta, w)\|_{\mathbf{N}}^{2}\right).
\label{ket-s10-01}
\end{equation}
\end{lemma}

\begin{proof}
Using \eqref{ps-02}--\eqref{ps-bd-1} and integration by parts, we
have
\begin{eqnarray*}
\frac{d}{dt}\|v\|_{2}^{2} &=&2\int_{\Omega}\langle v, \Delta
v\rangle\,dx -2\int_{\Omega}\langle v, \nabla p\rangle\,dx
-2\int_{\Omega}\langle v, (\beta+w)\cdot\nabla (v+w)\rangle\,dx  \\
&\leq &2 \|v\|_{H^{2}}^{2}+2\|v\|_{2}\|\nabla
p\|_{2}+2\|v\|_{2}\|\beta +w\|_{\infty}\|\nabla(v+w)\|_{2}.
\end{eqnarray*}
Using the Sobolev embedding, together with the $L^{2}$-estimate for
$p$, we have
\begin{equation}
\frac{d}{dt}\|v\|_{2}^{2}\leq C\|\beta+w\|_{H^{2}}\left( F+1\right)
\|\nabla w\|_{2}. \label{es29-841}
\end{equation}

\medskip
Next we show the following inequality:
\begin{equation}
\frac{d}{dt}\|v_{t}\|^{2}\leq C\left(1+\|\beta
+w\|_{\mathbf{N}}^{2}\right) \left( F+1+\|w_{t}\|_{2}^{2} +\|\nabla
w\|_{2}^{2}\right). \label{es29-842}
\end{equation}
%Let us prove (\ref{es29-842}) for example.
By differentiating (\ref{ps-02}), we obtain
\begin{equation}
\partial_tv_{t}=\Delta v_{t}-(\beta+w)\cdot\nabla(v_{t}+w_{t})
-(\beta_{t}+w_{t})\cdot\nabla(v+w)-\nabla p_{t}, \label{t29-01}
\end{equation}
and $v_{t}$ also satisfies the absolute boundary condition
\eqref{absolute-1}. Since $p_{t}$ is the time-derivative of $p$, it
solves the Neumann problem:
\begin{equation}
\left\{
\begin{array}{ll}
\Delta
p_{t}=-\nabla\cdot\big((\beta_{t}+w_{t})\cdot\nabla(\beta+w)\big)
-\nabla\cdot\big((\beta +w)\cdot\nabla(\beta_{t}+w_{t})\big), \\
\left. \partial_\nu p_{t}\right|_{\Gamma}=2\, \pi
\left(\beta_{t}+w_{t}, \beta +w\right).
\end{array}
\right.
\label{t29-02}
\end{equation}

According to the Solonnikov estimate \eqref{es29-04},
\begin{eqnarray}
\|\nabla p_{t}\| &\leq &\|(\beta_{t}+w_{t})\cdot\nabla(\beta+w)
+\nabla\cdot\left((\beta +w)\cdot\nabla(\beta_{t}+w_{t})\right)\|_{2} \notag \\
&\leq &\|\beta_{t}+w_{t}\|_{4}\|\nabla(\beta +w)\|_{4}
+\|\beta+h\|_{\infty}\|\nabla \beta_{t}+\nabla w_{t}\|_{2}  \notag \\
&\leq &C\|\beta_{t}+w_{t}\|_{H^{1}}\|\beta +w\|_{H^{2}},
\label{s29-91}
\end{eqnarray}
where the second inequality follows from the H\"{o}lder inequality,
and the last one follows from the Sobolev imbedding: $\|T\|_{4}\leq
C \|T\|_{H^{1}}$ and $\|T\|_{\infty }\leq C\|T\|_{H^{2}}$ in $\R^3$.
Using (\ref{t29-01}) yields
\begin{eqnarray*}
\frac{d}{dt}\|v_{t}\|_{2}^{2}=&&2\int_{\Omega}\langle v_{t}, \Delta
v_{t}\rangle\,dx -2\int_{\Omega}\langle v_{t}, (\beta+w)\cdot\nabla
(v_{t}+w_{t})\rangle\, dx\\
&& -2\int_{\Omega}\langle v_{t}, \nabla p_{t}\rangle\,dx
-2\int_{\Omega}\langle v_{t}, (\beta _{t}+w_{t})\cdot\nabla
(v+w)\rangle\,dx.
\end{eqnarray*}
Then, using the H\"{o}lder inequality, we have
\begin{eqnarray*}
\frac{d}{dt}\|v_{t}\|_{2}^{2} &\leq &-2\|\nabla
v_{t}\|_{2}^{2}-2\int_{\Gamma}\pi(v_{t},v_{t})\, dS
+\|\beta +w\|_{\infty}\|v_{t}\|_{2}\|\nabla v_{t}+\nabla w_{t}\|_{2} \\
&&+2\|v_{t}\|_{2}\|\nabla
p_{t}\|_{2}+\|\beta_{t}+w_{t}\|_{4}\|v_{t}\|_{4}\|\nabla v+\nabla
w\|_{2},
\end{eqnarray*}
where we have used (\ref{07june03-04}) and (\ref{s28-09}). According
to the trace imbedding theorem (Theorem 1.5.1.10 of Grisvard
\cite{gri1}, page 41),
\begin{equation*}
\big|\int_{\Gamma }\pi (v_{t},v_{t})\, dS\big| \leq
C\|v_{t}\|_{H^{1}}^{2}.
\end{equation*}
Thus we  establish
\begin{eqnarray}
\frac{d}{dt}\|v_{t}\|^{2} &\leq
&C\|v_{t}\|^{2}+2\|v_{t}\|_{2}\|\nabla
p_{t}\|_{2}+C\|\beta +w\|_{H^{2}}^{2}\|v_{t}+w_{t}\|_{2}^{2}  \notag \\
&&+C\|\beta_{t}+w_{t}\|_{H^{1}}\|v_{t}\|_{H^{1}}\|\nabla v+\nabla
w\|_{2},
\label{t29-07}
\end{eqnarray}
which implies (\ref{es29-842}).

Similarly, $\omega =\nabla\times v$ evolves according to the
vorticity equation:
\begin{equation}\label{t29-23}
\left\{\begin{array}{ll}
\partial_t \omega =\Delta \omega -\nabla
\times ((\beta+ w)\cdot\nabla (v+w)),\\
\left. \omega^{\Vert}\right|_{\Gamma}=0,
\end{array}\right.
\end{equation}
and $\omega_{t}$ satisfies the
evolution equation:
\begin{equation}\label{t29-06}
\left\{\begin{array}{ll}
\partial_t\omega_{t}=\Delta\omega_{t}-\nabla
\times \left( (\beta _{t}+w_{t})\cdot\nabla(v+w)+(\beta
+w)\cdot\nabla(v_{t}+w_{t})\right),\\
\left. \omega_{t}^{\Vert}\right|_{\Gamma}=0.
\end{array}\right.
\end{equation}
Moreover, $\nabla\cdot\omega_{t}=0$. Hence, integration by parts
yields
\begin{eqnarray}
&&\frac{d}{dt}\|\omega_{t}\|_{2}^{2}\nonumber\\
&&=2\int_{\Omega}\langle\omega_{t}, \Delta\omega_{t}\rangle\,dx
-2\int_{\Omega}\langle\omega_{t}, \nabla \times
\big((\beta_{t}+w_{t})\cdot\nabla(v+w)+(\beta+ w)\cdot\nabla
(v_{t}+w_{t})\big) \rangle\,dx
 \notag \\
&&=-2\int_{\Omega}|\nabla \times \omega_{t}|^{2}\,dx
-2\int_{\Omega}\langle \nabla\times\omega_{t},
(\beta_{t}+w_{t})\cdot\nabla (v+w)
+(\beta+w)\cdot\nabla(v_{t}+w_{t})\rangle\,dx  \notag \\
&&\leq
\frac{1}{2}\|(\beta_{t}+w_{t})\cdot\nabla(v+w)+(\beta+w)\cdot\nabla
(v_{t}+w_{t})\|_{2}^{2}, \label{t29-09}
\end{eqnarray}
which yields
\begin{equation}
\frac{d}{dt}\|\omega_{t}\|_{2}^{2}\leq C\|\beta+w\|_{\mathbf{N}}^{2}
\left(F+\|(\nabla w, \nabla w_{t})\|_{2}^{2}\right).
\label{wt-est-1}
\end{equation}

\medskip
Next we consider $\|d_{t}\|_{2}^{2}$ to show that
\begin{equation}
\frac{d}{dt}\|d_{t}\|_{2}^{2}\leq
C\|\beta+w\|_{\mathbf{N}}^{2}\left(F+\|\beta\|_{\mathbf{N}}^{2}\right).
\label{s29-102}
\end{equation}
Indeed, differentiating
equation (\ref{t29-21}) for $d$ to obtain
\begin{equation}
\partial_td_{t}=\Delta d_{t}+\nabla\cdot\big((\beta_{t}+w_{t})\cdot\nabla (\beta-v)
+(\beta+w)\cdot\nabla(\beta_{t}-v_{t})\big), \label{t29-04}
\end{equation}
subject to the Neumann boundary condition:
\begin{equation}
\left. \partial_\nu d_{t}\right|_{\Gamma} =\pi(\beta_{t}+w_{t},
\beta-v)+\pi(\beta+w, \beta_{t}-v_{t}).
\label{t29-05}
\end{equation}
It follows that
\begin{eqnarray*}
\frac{d}{dt}\|d_{t}\|_{2}^{2}=2\int_{\Omega}\langle d_{t}, \Delta
d_{t}\rangle\,dx +2\int_{\Omega}\langle d_{t},
\nabla\cdot\left((\beta_{t}+w_{t})\cdot\nabla(\beta-v)
+(\beta+w)\cdot\nabla(\beta_{t}-v_{t})\right)\rangle\,dx.
\end{eqnarray*}
Performing integrating by parts for the two integrals on the
right-hand side and noting that the boundary integrals cancel out,
we have
\begin{eqnarray*}
\frac{d}{dt}\|d_{t}\|_{2}^{2}&=&-2\int_{\Omega}|\nabla d_{t}|^{2}dx
+2\int_{\Omega}\langle \nabla d_{t}, (\beta_{t}+w_{t})\cdot\nabla(\beta-v)\rangle\,dx\\
&&+2\int_{\Omega}\langle\nabla d_{t}, (\beta +w)\cdot\nabla(\beta_{t}-v_{t})\rangle\,dx \\
&\leq &C \|(\beta_{t}+w_{t})\cdot\nabla (\beta-v)\|_{2}^{2}
+\|(\beta+w)\cdot\nabla(\beta_{t}-v_{t})\|_{2}^{2} \\
&\leq &C \|\beta_{t}+w_{t}\|_{4}^{2}\|\nabla \beta-\nabla
v\|_{4}^{2}+\|\beta +w\|_{\infty}^{2}\|\nabla\beta_{t}-\nabla v_{t}\|_{2}^{2}\\
&\leq &C\|\beta
+w\|_{\mathbf{N}}^{2}\big(\|\beta\|_{H^{2}}^{2}+\|v\|_{H^{2}}^{2}
+\|\nabla \beta_{t}-\nabla v_{t}\|_{2}^{2}\big),
\end{eqnarray*}
and (\ref{s29-102}) follows.

Next we handle the second-order derivative. That is, we need to
bound $\frac{d}{dt}\|\nabla g(t,\cdot)\|_{2}^{2}$,
$\frac{d}{dt}\|\nabla q(t,\cdot)\|_{2}^{2}$, and
$\frac{d}{dt}\|\psi(t,\cdot)\|_{L^{2}}^{2}$. We handle them one by
one.

Since $g$ satisfies (\ref{sa29-04}), integration by parts yields
\begin{eqnarray*}
\frac{d}{dt}\|\nabla g\|_{2}^{2}=-2\int_{\Omega}(\Delta g)^{2}\,dx
+2\int_{\Omega}\langle\nabla q_{t}, \nabla g\rangle\,dx \leq
\|\nabla q_{t}\|_{2}\|\nabla g\|_{2}.
\end{eqnarray*}
While, $q_{t}$ solves the Poisson equation:
\begin{equation*}
\Delta q_{t}=-\nabla\cdot\big((\beta_{t}+w_{t})\cdot\nabla
(\beta-v)+(\beta+w)\cdot\nabla(\beta_{t}-v_{t})\big),
\end{equation*}
subject to
\begin{eqnarray*}
\left. \partial_\nu q_{t}\right|_{\Gamma}
&=&\pi(\beta_{t}+w_{t}, \beta-v)+\pi(\beta +w, \beta_{t}-v_{t}) \\
&=&-\langle \nabla\cdot\big((\beta +w)\cdot\nabla(\beta-v)\big),\,
\nu\rangle.
\end{eqnarray*}
Thus, according to the Solonnikov estimate \eqref{es29-04},
\begin{eqnarray}
\|\nabla q_{t}\|_{2}&\leq &\|(\beta_{t}+w_{t})\cdot\nabla(\beta-u)
+(\beta+w)\cdot\nabla(\beta_{t}-v_{t})\|_{2} \notag \\
&\leq &\|\beta+w\|_{\mathbf{N}}\big(\|\beta\|_{H^{2}}+\|v\|_{2}
+\|\nabla(\beta_{t}-v_{t})\|_{2}\big)   \notag \\
&\leq &\|\beta+w\|_{\mathbf{N}}\left(\|\beta\|_{\mathbf{N}}
+\sqrt{F}\right), \label{s03-022}
\end{eqnarray}
and hence
\begin{equation*}
\frac{d}{dt}\|\nabla g(t,\cdot)\|_{2}^{2} \leq
2\|\beta+w\|_{\mathbf{N}}\left(\|\beta\|_{\mathbf{N}}+\sqrt{F}\right)
\sqrt{F}.
\end{equation*}

To estimate $\frac{d}{dt}\|\nabla q(t,\cdot)\|_{2}^{2}$, we begin
with
\begin{eqnarray*}
\frac{d}{dt}\|\nabla q(t, \cdot)\|_{2}^{2} =2\int_{\Omega}\langle
\nabla q_{t}, \nabla q\rangle\,dx =-2\int_{\Omega}q\Delta q_{t}\,dx
+2\int_{\Gamma}q\,
\partial_\nu q_{t}\, dS.
\end{eqnarray*}
Using the boundary condition for $q_{t}$ and integrating by parts again, we
obtain
\begin{eqnarray*}
\frac{d}{dt}\|\nabla q(t,\cdot)\|_{2}^{2}
&=&2\int_{\Omega}q\nabla\cdot\big((\beta_{t}+w_{t})\cdot\nabla(\beta-v)
   +(\beta+w)\cdot\nabla(\beta_{t}-v_{t})\big)\,dx  \\
&&+2\int_{\Gamma}q\big(\pi(\beta_{t}+w_{t}, \beta-v)
   +\pi(\beta+w, \beta_{t}-v_{t})\big)\,dS\\
&=&-2\int_{\Omega}\langle\nabla q,
(\beta_{t}+w_{t})\cdot\nabla(\beta-v)+(\beta+w)\cdot\nabla(\beta_{t}-v_{t})\rangle\,dS.
\end{eqnarray*}
Hence, we have
\begin{equation}
\frac{d}{dt}\|\nabla q\|_{2}^{2} \leq C\|\beta+w\|_{\mathbf{N}}\big(
\|\beta\|_{\mathbf{N}}+\sqrt{F}\big) \sqrt{F}. \label{s03-41}
\end{equation}

\medskip
Finally, we  establish a differential inequality for
$\|\psi\|_{2}^{2}$. Recall that the vorticity $\omega $ evolves
according to the parabolic equation \eqref{t29-23}.
Taking the curl operation $\nabla\times$ both sides of equation
\eqref{t29-23}, we obtain the evolution equation for $\psi$:
\begin{equation}
\partial_t\psi
=\Delta\psi
-\nabla\times\nabla\times\left((\beta+w)\cdot\nabla(v+w)\right),
\label{s27-41}
\end{equation}
so that
\begin{equation*}
\partial_t(|\psi|^{2})
=2\langle\Delta\psi, \psi\rangle
 -2\langle\nabla\times\nabla\times
 \left((\beta+w)\cdot\nabla(v+w)\right), \psi\rangle.
\end{equation*}
Integrating over $\Omega$ and performing integration by parts, one
then obtains
\begin{eqnarray}
\frac{d}{dt}\|\psi(t,\cdot)\|_{2}^{2}&=&2\int_{\Omega}\langle
\Delta\psi, \psi\rangle\,dx -2\int_{\Omega}\langle\nabla\times\nabla
\times \big((\beta+w)\cdot\nabla(v+w)\big), \psi\rangle\,dx   \notag \\
&=&-2\int_{\Omega}|\nabla\times\psi|^{2}\,dx
  -2\int_{\Omega}\langle\nabla\times\big((\beta+w)\cdot\nabla(v+w)\big),
  \nabla\times\psi\rangle\,dx \notag \\
&&+2\int_{\Gamma}\langle\psi\times(\nabla\times\psi),
    \mathbf{\nu}\rangle\,dS
   -2\int_{\Gamma}\langle\nabla\times\big((\beta+w)\cdot\nabla(v+w)\big)\times\psi,
    \nu\rangle\,dS,\notag\\
&& \label{s28-06}
\end{eqnarray}
where we have used the fact that $\nabla\cdot\psi=0$. Now we have to
handle the last two boundary integrals. The vector identity:
\begin{equation*}
\nabla \times \psi
=\nabla\times\left(\nabla\times\omega\right)=-\Delta \omega,
\end{equation*}
yields
\begin{eqnarray*}
\langle\psi\times(\nabla\times\psi), \mathbf{\nu}\rangle
=\langle\Delta \omega \times\psi, \mathbf{\nu}\rangle
=\langle\left(\Delta\omega\right)^{\Vert}\times\psi,
\mathbf{\nu}\rangle.
\end{eqnarray*}
However, since $\left. \omega^{\Vert}\right|_{\Gamma}=0$, it follows
from the vorticity equation \eqref{t29-23}
that
\begin{equation}
\left. \left(\Delta\omega\right)^{\Vert}\right|_{\Gamma}
=\big(\nabla\times\big((\beta+w)\cdot\nabla(v+w)\big)\big)^{\Vert}.
\label{s27-11}
\end{equation}
Therefore, the two boundary integrals sum up to zero. Hence, we have
\begin{equation}
\frac{d}{dt}\|\psi\|_{2}^{2}=-2\int_{\Omega}|\nabla\times\psi|^{2}\,dx
-2\int_{\Omega}\langle\nabla\times\big((\beta+w)\cdot\nabla(v+w)\big),
\nabla\times\psi\rangle\,dx. \label{s29-0031}
\end{equation}
To estimate the second integral, we consider $\alpha =\nabla \times
(X\cdot\nabla Y)$, where $X$ and $Y$ are two vector fields. Since
\begin{equation*}
(X\cdot\nabla Y)^{k}=X^{i}\nabla _{i}Y^{k},
\end{equation*}
then
\begin{eqnarray*}
\alpha^{j}
&=&\frac{1}{2}\varepsilon_{jkl}\left(\nabla_{k}(X^{a}\nabla_{a}Y^{l})
   -\nabla _{l}(X^{a}\nabla_{a}Y^{k})\right)\\
&=&\varepsilon_{jkl}(\nabla_{k}X^{a})(\nabla_{a}Y^{l})
 +\frac{1}{2}\varepsilon_{jkl}X^{a}
  \big(\nabla_{k}\nabla_{a}Y^{l}-\nabla_{l}\nabla_{a}Y^{k}\big)\\
&=&\varepsilon_{jkl}(\nabla_{k}X^{a})(\nabla_{a}Y^{l})
  +X^{a}\nabla_{a}\left(\nabla\times Y\right)^{j}.
\end{eqnarray*}
That is,
\begin{equation*}
\nabla\times(X\cdot\nabla
Y)=\varepsilon_{jkl}(\nabla_{k}X^{a})(\nabla_{a}Y^{l})
     +X\cdot\nabla\left(\nabla\times Y\right),
\end{equation*}
so that
\begin{equation*}
|\nabla\times(X\cdot\nabla Y)| \leq |\nabla X||\nabla
Y|+|X||\nabla\left(\nabla\times Y\right)|.
\end{equation*}
Using this inequality in (\ref{s29-0031}), one deduces that
\begin{eqnarray*}
\frac{d}{dt}\|\psi(t,\cdot)\|_{L^{2}}^{2} &\leq
&-2\int_{\Omega}|\nabla\times\psi|^{2}\, dx +
2\int_{\Omega}|\nabla\times\psi|\,
|\nabla(\beta+w)|\,|\nabla(v+w)|\,dx
\\
&&+2\int_{\Omega}|\nabla\times\psi|\,|\beta+w|\,
|\nabla(\omega+\nabla\times w)|\,dx \\
&\leq &-2\|\nabla\times\psi\|_{2}^{2}+2\|\nabla\times\psi\|_{2}
  \|\nabla(\beta+w)\|_{4}\|\nabla(v+w)\|_{4} \\
&&+2\|\nabla\times\psi\|_{2}\|\beta +w\|_{\infty}\|\nabla\omega
+\nabla(\nabla\times w)\|_{2} \\
&\leq &-2\|\nabla\times\psi\|_{2}^{2}+C\|\nabla\times\psi\|_{2}
\|\beta+w\|_{H^{2}}\|v+w\|_{H^{2}},
\end{eqnarray*}
which implies that
\begin{equation}
\frac{d}{dt}\|\psi\|_{L^{2}}^{2}\leq C\|\beta
+w\|_{\mathbf{N}}^{2}\left(F+\|w\|_{H^{2}}^{2}\right).
\label{s03-42}
\end{equation}

Combining all the estimates yields (\ref{ket-s10-01}). We thus
complete the proof of Theorem \ref{them01}.
\end{proof}

\subsection{Local strong solutions}
For $T>0$, denote $\mathbf{W}_{T}$ the space of all vector fields
$\beta (t,x)$ that satisfy conditions (i)--(iii) in Theorem
\ref{them01} with the uniform norm:
\begin{equation*}
\|\beta\|_{\mathbf{W}_{T}}=\sup_{0\le t\leq T}\sqrt{\|\beta
(t,\cdot)\|_{H^{2}}^{2}+\left\|\partial_t\beta
(t,\cdot)\right\|_{H^{1}}^{2}}.
\end{equation*}
Then $\mathbf{W}_{T}$ is a Banach space. According to
(\ref{s29-071}), there is a constant $C>0$ depending only on $\Omega
$ such that
\begin{eqnarray}
\|V(\beta)\|_{\mathbf{W}_{T}}
&\leq &C\sqrt{C_{0}}\,e^{C(1+\|(\beta, w)\|_{\mathbf{W}_{T}}^{2}) T}  \notag \\
&&+C\sqrt{T}\,e^{C(1+\|(\beta, w)\|_{\mathbf{W}_{T}}^{2}) T}
   \big(1+\|(\beta,w)\|_{\mathbf{W}_{T}}^{2}\big).
\label{s31-31}
\end{eqnarray}
Equipped with the \emph{apriori} estimate (\ref{s29-071}), we are
now in a position to establish the existence of a local (in time)
strong solutions for the initial-boundary value problem with
nonhomogeneous boundary conditions.

\begin{theorem}
Given $C_{0}>1$, there exists $T>0$ depending only on $C_{0}$, the
viscosity coefficient $\mu >0$, and the domain $\Omega$ such that,
for any given $u_{0}\in H^{2}(\Omega)$ which satisfies the absolute
boundary conditions \eqref{kinematic-1}--\eqref{vorticity-1},
$\|u_{0}\|_{H^{2}}\leq C_{0}$, and $\nabla\cdot u_{0}=0$, there
exists a strong solution $u(t,x)$ of \eqref{1.1}--\eqref{1.3} and
\eqref{kinematic-1}--\eqref{vorticity-1} with the form
$u(t,x)=v(t,x)+w(t,x), 0\le t\leq T,$ such that

{\rm (i)} For every $0\le t\leq T$, $v(t,\cdot)\in H^{2}(\Omega)$
and $\partial_t v(t,\cdot)\in H^{1}(\Omega)$;

{\rm (ii)} For every $0\le t\leq T$, $\left.
v^{\bot}\right|_{\Gamma}=0$ and $\left. (\nabla\times
v)^{\Vert}\right|_{\Gamma}=0$;

{\rm (iii)} $t\rightarrow \|v(t,\cdot)\|_{H^{2}}^{2}+\|\partial_t
v(t,\cdot)\|_{H^{1}}^{2}$ is continuous;

{\rm (iv)} $v|_{t=0}=0$, and $v$ satisfies the Navier-Stokes
equations
\begin{equation*}
\partial_t v+(v+w)\cdot\nabla (v+w)=\Delta v-\nabla p, \qquad \nabla\cdot
v=0,
\end{equation*}
where $p$ solves the Poisson equation for each $t\in (0, T]$:
\begin{equation*}
\Delta p=-\nabla\cdot\left((v+w)\cdot\nabla (v+w)\right),
\qquad\left. \partial_\nu p\right|_{\Gamma}=\pi(v+w,v+w),
\end{equation*}
and $w$ solves the initial-boundary value problem \eqref{in-stokes}
for the unsteady Stokes equations.
\end{theorem}

\begin{proof}
{\it Step 1}. If $\beta_{1}(t,x)$ and $\beta_{2}(t,x)$, $0\le t\le
T$, are two vector fields which satisfy conditions (i)-(iv) in
Theorem \ref{them01}, then $U:=v_{1}-v_{2}$ with
$v_{k}=V(\beta_{k}), k=1,2,$ satisfies the linear parabolic
equations:
\begin{equation*}
\partial_tU=\Delta U-(\beta_{1}+w)\cdot\nabla U
-\nabla P -(\beta_{1}-\beta_{2})\cdot\nabla (v_{2}+w), \qquad
U|_{t=0}=0,
\end{equation*}
subject to $\left. U^{\bot}\right|_{\Gamma}=0$ and $\left. \left(
\nabla \times U\right)^{\Vert}\right|_{\Gamma}=0$, where
$P=p_{\beta_{1}}-p_{\beta_{2}}$ is the unique solution to
\begin{equation*}
\Delta P=-\nabla\cdot\big( (\beta_{1}+w)\cdot\nabla(\beta_{1}+w)
-(\beta_{2}+w)\cdot\nabla(\beta_{2}+w)\big)
\end{equation*}
such that
\begin{equation*}
\left. \partial_\nu P\right|_{\Gamma} =\pi\left(\beta_{1}+w,
\beta_{1}+w\right) -\pi \left(\beta_{2}+w, \beta_{2}+w\right).
\end{equation*}
Applying the same arguments used in the proof of Theorem
\ref{them01}, we obtain that there exists $C>0$ depending only on
the domain $\Omega$ such that the following inequality holds:
\begin{eqnarray}
&&\|V(\beta_{1})-V(\beta_{2})\|_{\mathbf{W}_{T}} \notag \\
&&\leq C\sqrt{T}e^{C(1+\|(\beta_{1}, \beta_{2},
w)\|_{\mathbf{W}_{T}}^{2}) T}
    \big(1+\|(V(\beta_{1}), \beta_{1}, \beta_{2},
w)\|_{\mathbf{W}_{T}}\big)\|\beta_{1}-\beta_{2}\|_{\mathbf{W}_{T}}.\quad
\label{lip-01-01}
\end{eqnarray}
It follows from (\ref{s31-31})--(\ref{lip-01-01}) that, for every
$K>\sqrt{CC_{0}}$, there exists $T>0$ depending only $K$, $C$, and
$\|w\|_{\mathbf{N}}$ such that, if $v\in \mathbf{W}_{T}$ with
$\|v\|_{\mathbf{W}_{T}}\leq K$, then $\|V(v)\|_{\mathbf{W}_{T}}\leq
K$ and
\begin{equation*}
\|V(v_1)-V(v_1)\|_{\mathbf{W}_{T}}\leq
\frac{2}{3}\|v_1-v_2\|_{\mathbf{W}_{T}}
\end{equation*}
for any $v_1,v_2\in \mathbf{W}_{T}$ such that $v_1(0,x)=v_2(0,x)=0$,
$\|v_i\|_{\mathbf{W}_{T}}\leq K$, $i=1,2$. Therefore, there is a
unique fixed point $v\in \mathbf{W}_{T}$ such that $V(v)=v$. Then
$u=v+w$, according to Theorem \ref{them01},
a solution of the nonhomogeneous initial-boundary value problem
\eqref{1.1}--\eqref{1.3} and
\eqref{kinematic-1}--\eqref{vorticity-1}  for the Navier-Stokes
equations.

\medskip
{\it Step 2. Incompressibility}:
If $v(t,\cdot ), 0\le t\le T$, is a fixed point of the velocity map
$V$, then $v$ is a solution to
\begin{equation}
\nabla\cdot v=0. \label{nss01}
\end{equation}

Setting $\beta =v$ in \eqref{ps-01}--\eqref{ps-02}
and taking divergence to obtain
\begin{equation*}
\partial_t (\nabla\cdot v)+\nabla\cdot\big((v+w)\cdot\nabla (v+w)\big)
=\Delta(\nabla\cdot v)-\Delta p_{v}
\end{equation*}
so that $d=\nabla\cdot v$ solves the heat equation:
\begin{equation*}
\partial_t d=\Delta d, \qquad d(0,\cdot)=0.
\end{equation*}
According to our assumptions:
\begin{equation}
\partial_t v+\left(v+w\right)\cdot\nabla \left(v+w\right) =\Delta
v-\nabla p_{v}.  \label{nssde1}
\end{equation}
Identifying the normal part of each term of equation (\ref{nssde1}),
together with the boundary condition: $\left.
v^{\bot}\right|_{\Gamma}=\left. w^{\bot}\right|_{\Gamma}=0$, we
conclude
\begin{eqnarray*}
\left(\Delta v\right)^{\bot}=\left(\nabla p_{v}\right)^{\bot}
 +\left(\left(v+w\right)\cdot\nabla \left(v+w\right)\right)^{\bot}
=0.
\end{eqnarray*}
On the other hand, by Lemma \ref{lems2}, together with the boundary
condition: $\left. \omega^{\Vert}\right|_{\Gamma}=0$, we have
\begin{equation*}
\left. \partial_\nu d\right|_{\Gamma} =\langle \Delta v,
\mathbf{\nu}\rangle =0.
\end{equation*}
By the uniqueness of the Neumann problem, $d(t,\cdot)=0$ for all
$t$.
\end{proof}

\section{Inviscid Limit in $\Omega\subset\R^n, n\ge 3$}

In this section we study the inviscid limit of the solutions to the
nonhomogeneous initial-boundary value problem
\eqref{1.1}--\eqref{1.3} and
\eqref{kinematic-1}--\eqref{vorticity-1}  for the Navier-Stokes
equations.

Let $u_{\mu}$ be the solution to the initial-boundary value problems
\eqref{1.1}--\eqref{1.3} and
\eqref{kinematic-1}--\eqref{vorticity-1} for the Navier-Stokes
equations for each $\mu>0$. Let $u$ be the solution to the
initial-boundary value problem \eqref{1.4}--\eqref{kinematic-1} for
the Euler equations.
Notice that all solutions $u_{\mu}$ subject to the same boundary
conditions: $\left. u_{\mu}^{\bot}\right|_{\Gamma}=0$ and $\left.
(\nabla\times u_{\mu})^{\Vert}\right|_{\Gamma}=\nabla^{\Gamma}\times
a$ for the given smooth vector field $a$,
while
the solution $u$ satisfies only the kinematic condition $\left.
u^{\bot}\right|_{\Gamma}=0$ which is independent of the viscosity
constant $\mu$.

\begin{theorem}
Let $a\in L^2([0,T]; L^2(\Gamma))$ be a smooth vector field on
$\Gamma$. Suppose that, for all $\mu\in (0, \mu_{0}]$, a unique
strong solution $u_{\mu}$ of problem \eqref{1.1}--\eqref{1.3} and
\eqref{kinematic-1}--\eqref{vorticity-1}
and  a unique strong solution $u\in H^2(\Omega)$ of  problem
\eqref{1.4}--\eqref{kinematic-1}
exist up to time $T^{\ast}>0$. Then there exists $C>0$ depending on
$\mu_0, T, \|a\|_{L^2([0,T]; L^2(\Gamma)}$, and $\|u\|_{H^2\cap
W^{1,\infty}(\Omega)}$, independent of $\mu$, such that, for any
$T\in [0, T^{\ast}]$,
\begin{equation}
\sup_{0\le t\leq T}\|u_{\mu}(t,\cdot )-u(t,\cdot)\|_{2}\le
C(T)\mu\rightarrow 0\text{ \ \ as }\mu \downarrow 0, \label{th01}
\end{equation}
and
\begin{equation*}
\int_{0}^{T}||\nabla \left( u_{\mu }-u\right) (s,\cdot
)||_{2}^{2}ds\leq C.
\end{equation*}
It follows that the solutions $u^\mu$ of problem
\eqref{1.1}--\eqref{1.3} and
\eqref{kinematic-1}--\eqref{vorticity-1} for the Navier-Stokes
equations converge to the unique solution $u(t,x)$ of problem
\eqref{1.4}--\eqref{kinematic-1} in $L^\infty([0,T]; L^2(\Omega))$.
\end{theorem}

\begin{proof}
Let $v_{\mu }=u_{\mu }-u$. Then $v_{\mu }$ satisfies the following
equations:
\begin{equation}
\left\{
\begin{array}{ll}
{\partial_t}v_{\mu} =\mu \Delta
v_{\mu}-\left(v_{\mu}+u\right)\cdot\nabla v_{\mu} -\nabla
P_{\mu}-v_{\mu}\cdot\nabla u+\mu\Delta u, \\
\nabla\cdot v_{\mu}=0\text{,}
\end{array}
\right.  \label{0-c-1}
\end{equation}
and the initial condition:
\begin{equation}\label{initial}
v_{\mu}(0,\cdot)=0,
\end{equation}
where $P_{\mu}=p_{\mu}-p$.

Since both $u_{\mu}$ and $u$ satisfy the kinematic condition
\eqref{kinematic-1}, so does $v_{\mu}$. Thus, by means of the energy
method, we obtain
\begin{eqnarray*}
\frac{d}{dt}\|v_{\mu}\|_{2}^{2} &=&2\mu
\int_{\Omega}v_{\mu}\cdot\Delta v_{\mu}\,dx
-\int_{\Omega}\left(v_{\mu}+u\right)\cdot\nabla
\big(|v_{\mu}|^{2}\big)\,dx
-2\int_{\Omega}\langle \nabla P_{\mu}, v_{\mu}\rangle\,dx \\
&&-2\int_{\Omega}\langle v_{\mu}\cdot\nabla u, v_{\mu}\rangle\, dx
+2\mu\int_{\Omega}\langle \Delta u, v_{\mu}\rangle\, dx.
\end{eqnarray*}
Integrating by parts in the first three integrals, one then deduces
that
\begin{eqnarray*}
\frac{d}{dt}\|v_{\mu}\|_{2}^{2} &=&-2\mu \int_{\Omega}|\nabla\times
v_{\mu}|^{2}\,dx +2\mu\int_{\Gamma}\langle v_{\mu}\times
\left(\nabla\times v_{\mu}\right), \nu\rangle \,dS\\
&&-2\int_{\Omega}\langle v_{\mu}\cdot\nabla u, v_{\mu}\rangle\,dx
+2\mu\int_{\Omega}\langle\Delta u, v_{\mu}\rangle\,dx \\
&=&-2\mu\|\nabla v_{\mu}\|_{2}^{2}-2\mu \int_{\Gamma}\pi(v_{\mu},
v_{\mu})\,dS+2\mu \int_{\Gamma}\langle v_{\mu}\times b, \nu\rangle\,dS \\
&&-2\int_{\Omega}\langle v_{\mu}\cdot\nabla u, v_{\mu}\rangle\,dx
+2\mu\int_{\Omega}\langle \Delta u, v_{\mu}\rangle\,dx,
\end{eqnarray*}
where $b=a-\nabla\times u$. Furthermore, we  use the following
estimate:
\begin{eqnarray*}
\int_{\Gamma}\langle v_{\mu}\times b,\nu \rangle\,dS \leq
\|b\|_{L^{2}(\Gamma)}\|v_{\mu}\|_{L^{2}(\Gamma)} \le
C\big(\|v_{\mu}\|_{L^{2}(\Gamma)}^{2}+\|a-\nabla\times
u\|_{L^{2}(\Gamma)}^{2}\big)
\end{eqnarray*}
to obtain
\begin{eqnarray}
\frac{d}{dt}\|v_{\mu}\|_{2}^{2} &\leq& -2\mu\|\nabla
v_{\mu}\|_{2}^{2}+2\mu C\left( \|v_{\mu}\|_{L^{2}(\Gamma)}^{2}
+\|a-\nabla\times u\|_{L^{2}(\Gamma)}^{2}\right)  \notag \\
&&+2\|\nabla u\|_{\infty}\|v_{\mu}\|_{2}^{2}+2\mu \|\Delta
u\|_{2}\|v_{\mu}\|_{2}\text{ .}  \label{p-01}
\end{eqnarray}
Finally, we use the Sobolev trace theorem:
\begin{equation*}
2C\|v_\mu\|^2_{L^2(\Gamma)} \le \|\nabla v_\mu\|_2^2
+\tilde{C}\|v_\mu\|_2^2, \qquad \|\nabla \times
u\|_{L^{2}(\Gamma)}^{2}\leq \|u\|^2_{H^2(\Omega)}
%\\nabla
%v_{\mu}\|_{2}^{2}+C_{1}\|v_{\mu}\|_{2}^{2}
\end{equation*}
to establish the differential inequality:
\begin{equation}
\frac{d}{dt}\|v_{\mu}\|_{2}^{2}+\mu \|\nabla v_{\mu}\|_{2}^{2} \le
C\left((\|\nabla u\|_{\infty}+\mu_0) \|v_{\mu}\|_{2}^{2}
+\mu ( \|a\|_{L^{2}(\Gamma)}^{2}+\|u\|_{H^2(\Omega)}^{2})\right)
\text{.} \label{p-02}
\end{equation}
The Gronwall inequality yields that, for $t\in [0, T^*]$,
\begin{equation}
\|v_{\mu}(t, \cdot)\|_{2}^{2} \leq
\mu C\int_{0}^{t}e^{C\int_{s}^{t}(\|\nabla
u(\tau,\cdot)\|_{\infty}+\mu (t-s))}\big(
\|b\|_{L^{2}(\Gamma)}^{2}+\|u(s,\cdot)\|_{H^2(\Omega)}^{2} \big)
ds:=C\mu \text{.} \label{p-04}
\end{equation}
Hence
\begin{equation*}
\int_{0}^{t}\|\nabla v_{\mu}(s, \cdot)\|_{2}^{2}ds\leq
C,
\end{equation*}
which imply the conclusions of the theorem.
\end{proof}

In order to ensure the convergence of $u_{\mu}$ to $u$ in the strong
sense (say, in $H^{2}(\Omega)$) up to the boundary, a necessary
condition is that $u$ must match with the boundary data $\left.
(\nabla \times u)^{\bot}\right|_{\Gamma}=a$.

\bigskip
\medskip \noindent
{\bf Acknowledgments.} Gui-Qiang Chen's research was supported in
part by the National Science Foundation under Grants DMS-0807551,
DMS-0720925, and DMS-0505473, and the Natural Science Foundation of
China under Grant NSFC-10728101. Zhongmin Qian's research was
supported in part by EPSRC grant EP/F029578/1. This paper was
written as part of  the international research program on Nonlinear
Partial Differential Equations at the Centre for Advanced Study at
the Norwegian Academy of Science and Letters in Oslo during the
academic year 2008--09.

\end{document}